\documentclass[12pt]{article}
\usepackage{caption}
\usepackage{amsmath}
\usepackage{amssymb}%
\usepackage{ifthen}
\usepackage{bm}
\usepackage{amsfonts}
\usepackage{multirow}
    \usepackage{latexsym}
\usepackage{amsthm}
\usepackage{mathrsfs}
\usepackage[dvips]{graphicx}
\usepackage{float}
\usepackage{varioref}
\usepackage[top=1in, bottom=1in, left=1in, right=1in]{geometry}
\usepackage{array}
\newcommand{\Dchaintwo}[4]{
\rule[-3\unitlength]{0pt}{8\unitlength}
\begin{picture}(14,5)(0,3)
\put(1,2){\ifthenelse{\equal{#1}{l}}{\circle*{2}}{\circle{2}}}
\put(2,2){\line(1,0){10}}
\put(13,2){\ifthenelse{\equal{#1}{r}}{\circle*{2}}{\circle{2}}}
\put(1,5){\makebox[0pt]{\scriptsize #2}}
\put(7,4){\makebox[0pt]{\scriptsize #3}}
\put(13,5){\makebox[0pt]{\scriptsize #4}}
\end{picture}}

\newcommand{\Dchainthree}[6]{
\rule[-3\unitlength]{0pt}{8\unitlength}
\begin{picture}(26,5)(0,3)
\put(1,2){\ifthenelse{\equal{#1}{l}}{\circle*{2}}{\circle{2}}}
\put(2,2){\line(1,0){10}}
\put(13,2){\ifthenelse{\equal{#1}{m}}{\circle*{2}}{\circle{2}}}
\put(14,2){\line(1,0){10}}
\put(25,2){\ifthenelse{\equal{#1}{r}}{\circle*{2}}{\circle{2}}}
\put(1,5){\makebox[0pt]{\scriptsize #2}}
\put(7,4){\makebox[0pt]{\scriptsize #3}}
\put(13,5){\makebox[0pt]{\scriptsize #4}}
\put(19,4){\makebox[0pt]{\scriptsize #5}}
\put(25,5){\makebox[0pt]{\scriptsize #6}}
\end{picture}}
\newcommand{\Dtriangle}[7]{
\rule[-3\unitlength]{0pt}{12\unitlength}
\begin{picture}(18,7)(0,3)
\put(4,4){\ifthenelse{\equal{#1}{l}}{\circle*{2}}{\circle{2}}}
\put(5,4){\line(1,0){8}}
\put(14,4){\ifthenelse{\equal{#1}{r}}{\circle*{2}}{\circle{2}}}
\put(4.4472,4.8944){\line(1,2){4.1056}}
\put(9,14){\ifthenelse{\equal{#1}{t}}{\circle*{2}}{\circle{2}}}
\put(13.5528,4.8944){\line(-1,2){4.1056}}
\put(2,3){\makebox[0pt][r]{\scriptsize #2}}
\put(9,17){\makebox[0pt]{\scriptsize #3}}
\put(16,3){\makebox[0pt][l]{\scriptsize #4}}
\put(6,9){\makebox[0pt][r]{\scriptsize #5}}
\put(12.5,9){\makebox[0pt][l]{\scriptsize #6}}
\put(9,1){\makebox[0pt]{\scriptsize #7}}
\end{picture}}

\newcommand{\Dchainfour}[8]{
\rule[-3\unitlength]{0pt}{5\unitlength}
\begin{picture}(38,5)(0,3)
\put(1,2){\ifthenelse{\equal{#1}{1}}{\circle*{2}}{\circle{2}}}
\put(2,2){\line(1,0){10}}
\put(13,2){\ifthenelse{\equal{#1}{2}}{\circle*{2}}{\circle{2}}}
\put(14,2){\line(1,0){10}}
\put(25,2){\ifthenelse{\equal{#1}{3}}{\circle*{2}}{\circle{2}}}
\put(26,2){\line(1,0){10}}
\put(37,2){\ifthenelse{\equal{#1}{4}}{\circle*{2}}{\circle{2}}}
\put(1,5){\makebox[0pt]{\scriptsize #2}}
\put(7,4){\makebox[0pt]{\scriptsize #3}}
\put(13,5){\makebox[0pt]{\scriptsize #4}}
\put(19,4){\makebox[0pt]{\scriptsize #5}}
\put(25,5){\makebox[0pt]{\scriptsize #6}}
\put(31,4){\makebox[0pt]{\scriptsize #7}}
\put(37,5){\makebox[0pt]{\scriptsize #8}}
\end{picture}}

\newcommand{\bigDchainfour}[8]{
\rule[-3\unitlength]{0pt}{5\unitlength}
\begin{picture}(42,5)(0,3)
\put(1,2){\ifthenelse{\equal{#1}{1}}{\circle*{2}}{\circle{2}}}
\put(2,2){\line(1,0){14}}
\put(17,2){\ifthenelse{\equal{#1}{2}}{\circle*{2}}{\circle{2}}}
\put(18,2){\line(1,0){10}}
\put(29,2){\ifthenelse{\equal{#1}{3}}{\circle*{2}}{\circle{2}}}
\put(30,2){\line(1,0){10}}
\put(41,2){\ifthenelse{\equal{#1}{4}}{\circle*{2}}{\circle{2}}}
\put(1,5){\makebox[0pt]{\scriptsize #2}}
\put(9,4){\makebox[0pt]{\scriptsize #3}}
\put(17,5){\makebox[0pt]{\scriptsize #4}}
\put(23,4){\makebox[0pt]{\scriptsize #5}}
\put(29,5){\makebox[0pt]{\scriptsize #6}}
\put(35,4){\makebox[0pt]{\scriptsize #7}}
\put(41,5){\makebox[0pt]{\scriptsize #8}}
\end{picture}}

\newcommand{\Dthreefork}[8]{
\rule[-9\unitlength]{0pt}{12\unitlength}
\begin{picture}(28,12)(0,9)
\put(2,10){\ifthenelse{\equal{#1}{l}}{\circle*{2}}{\circle{2}}}
\put(3,10){\line(1,0){10}}
\put(14,10){\ifthenelse{\equal{#1}{m}}{\circle*{2}}{\circle{2}}}
\put(15,10){\line(1,1){7}}
\put(15,10){\line(1,-1){7}}
\put(22,18){\ifthenelse{\equal{#1}{t}}{\circle*{2}}{\circle{2}}}
\put(22,2){\ifthenelse{\equal{#1}{b}}{\circle*{2}}{\circle{2}}}
\put(2,12){\makebox[0pt]{\scriptsize #2}}
\put(8,11){\makebox[0pt]{\scriptsize #3}}
\put(14,12){\makebox[0pt]{\scriptsize #4}}
\put(19,16){\makebox[0pt][r]{\scriptsize #5}}
\put(19,4){\makebox[0pt][r]{\scriptsize #6}}
\put(24,17){\makebox[0pt][l]{\scriptsize #7}}
\put(24,2){\makebox[0pt][l]{\scriptsize #8}}
\end{picture}}

\newcommand{\Drightofway}[9]{
\rule[-9\unitlength]{0pt}{12\unitlength}
\begin{picture}(28,12)(0,9)
\put(2,10){\ifthenelse{\equal{#1}{l}}{\circle*{2}}{\circle{2}}}
\put(3,10){\line(1,0){10}}
\put(14,10){\ifthenelse{\equal{#1}{m}}{\circle*{2}}{\circle{2}}}
\put(15,10){\line(1,1){7}}
\put(15,10){\line(1,-1){7}}
\put(22,18){\ifthenelse{\equal{#1}{t}}{\circle*{2}}{\circle{2}}}
\put(22,2){\ifthenelse{\equal{#1}{b}}{\circle*{2}}{\circle{2}}}
\put(22,3){\line(0,1){14}}
\put(2,12){\makebox[0pt]{\scriptsize #2}}
\put(8,11){\makebox[0pt]{\scriptsize #3}}
\put(14,12){\makebox[0pt]{\scriptsize #4}}
\put(19,16){\makebox[0pt][r]{\scriptsize #5}}
\put(19,4){\makebox[0pt][r]{\scriptsize #6}}
\put(24,18){\makebox[0pt][l]{\scriptsize #7}}
\put(23,10){\makebox[0pt][l]{\scriptsize #8}}
\put(24,1){\makebox[0pt][l]{\scriptsize #9}}
\end{picture}
}

\newcounter{cthm}%

\newlength{\mpb}

\newcommand{\Aut }{\mathrm{Aut}}
\newcommand{\Hom }{\mathrm{Hom}}
\newcommand{\End }{\mathrm{End}}

\newcommand{\cB }{\mathcal{B}}
\newcommand{\cC }{\mathcal{C}}
\newcommand{\cR }{\mathcal{R}}

\newcommand{\cD }{\mathcal{D}}

\newcommand{\cI }{\mathcal{I}}

\newcommand{\cX }{\mathcal{X}}
\newcommand{\cY }{\mathcal{Y}}
\newcommand{\cW }{\mathcal{W}}

\newcommand{\gr }{\mathrm{gr}}

\newcommand{\lact }{.}

\newcommand{\ndN }{\mathbb{N}}

\newcommand{\ndZ }{\mathbb{Z}}

\newcommand{\ot }{\otimes }

\newcommand{\PBW }{Poincar\'e--Birkhoff--Witt }

\newcommand{\roots }{\boldsymbol{\Delta }}

\newcommand{\YD }{Yetter--Drinfel'd }

\newcommand{\ydH }{ {}^H_H\mathcal{YD}}
\newcommand{\ydD }{ {}^{G}_{G}\mathcal{YD}}

\newcommand{\id}{\mathrm{id}}
\newcommand{\al }{\alpha }

\newcommand{\btxandshort}[1]{and}%
\newcommand{\btxpagesshort}[1]{pp.}%
\newcommand{\Btxinshort}[1]{In}%
\newcommand{\btxphdthesis}[1]{phd-thesis}%
\newcommand{\btxeditorshort}[1]{Ed.}%
\newcommand{\btxeditorsshort}[1]{Eds.}%
\newcommand{\btxvolumeshort}[1]{vol.}%
\newcommand{\btxofseriesshort}[1]{ser.}%
\newcommand{\ffg }{\mathcal{F}_\theta^{G} }   
\newcommand{\fiso }{\mathcal{X}_\theta }      
\newcommand{\rersys }[1]{\boldsymbol{\Delta }{}^{#1\,\mathrm{re}}}
\newcommand{\rsys }{\boldsymbol{\Delta }}
\newcommand{\ad }{\mathrm{ad}}
\newcommand{\hSL}[1]{\widehat{s\ell}(#1)}
\newtheorem{theorem}{\bf Theorem}[section]
\newtheorem{lemma}[theorem]{\bf Lemma}
\newtheorem{prop}[theorem]{\bf Proposition}
\newtheorem{cor}[theorem]{\bf Corollary}

\newtheorem{question}[theorem]{\bf Question}
\newtheorem{defn}[theorem]{\bf Definition}

\newtheorem{rem}[theorem]{\bf Remark}

\begin{document}

\title{Rank 4 finite-dimensional Nichols algebras of diagonal type in positive characteristic}

\author{
Jing Wang\thanks{supported by the fundamental Research Funds for the Central Universities (No.BLX201721) and the National Natural Science Foundation of China (No.11901034) }
\\wang\_jing619@163.com
\\Department of Mathematics
\\School of Science
\\Beijing Forestry University
\\100083 Beijing, China}
\date{}
\maketitle

\begin{abstract}
Nichols algebras are fundamental objects in the construction of quantized enveloping algebras and in the classification of pointed Hopf algebras by lifting method of Andruskiewitsch and Schneider. Arithmetic root systems are invariants of Nichols algebras of diagonal type with a certain finiteness property.
In the present paper, all rank 4 Nichols algebras of diagonal type with a finite arithmetic root system over fields of arbitrary characteristic are classified. Our proof uses the classification of the finite arithmetic root systems of rank 4.

\end{abstract}

\section*{Introduction}
The theory of Nichols algebras is relatively young, but it has interesting applications to other research fields such as Kac-Moody Lie superalgebras~\cite{inp-Andr14} and conformal field theory~\cite{Semi-2011,Semi-2012,Semi-2013}. Besides, it plays an important role in quantum groups~\cite{inp-Andr02,a-AndrGr99,inp-AndrSchn02, a-Schauen96}.

The theory of Nichols algebras is motivated by the Hopf algebra theory. In any area of mathematics the classification of all objects is very important. In Hopf algebra theory, the classification of all finite dimensional Hopf algebras is a tough question\cite{inp-Andr02}. The structure of Nichols algebras appears naturally in the classification of pointed Hopf algebras in the following way.
Given a Hopf algebra $H$, consider its coradical filtration $$H_0\subset H_1\subset \ldots $$ such that $H_0$ is a Hopf subalgebra of $H$ and the associated graded coalgebra $$\gr H=\bigoplus _iH_i/H_{i-1}.$$ Then $\gr H$  is a graded Hopf algebra, since the coradical $H_0$ of $H$ is a Hopf
subalgebra. In addition, consider a projection $\pi: \gr H\to H_{0}$; let $R$ be the algebra of coinvariants of $\pi$. Then, by a result of Radford and Majid, $R$ is a braided Hopf algebra and $\gr H$ is the bosonization (or
biproduct) of $R$ and $H_{0}$: $\gr H \simeq  R\# H_{0}$.
The principle of the "Lifting method" introduced in~\cite{a-AndrSchn98,inp-AndrSchn02} is first to study $R$, then to transfer the information to $\gr H$ via bosonization, and finally to lift to $H$.
The braided Hopf algebra $R$ is generated by the vector space $V$ of $H_0$-coinvariants of $H_1/H_0$, namely \textit{Nichols algebra} $\cB(V)$ generated by $V$ \cite{a-AndrSchn98} in commemoration of W.~Nichols who started to study these objects as bialgebras of type one in~\cite{n-78}.
Nichols algebras can be described in many different but alternative ways, see for example~\cite{l-2010,maj05,Rosso98,w1987,w1989}.

The crucial step to classify pointed Hopf algebras is to determine all Nichols algebras $\cB(V)$ is finite dimensional.
 N.~Andruskiewitsch stated the following question.
\begin{question} \textbf{(N.~Andruskiewitsch~\cite{inp-Andr02})}\label{quse:classification} \it ~
 Given a braiding matrix $(q_{ij})_{1\leq i,j\leq \theta} $ whose entries are roots of $1$, when $\cB(V)$ is finite-dimensional, where $V$ is a vector space with basis $x_1,\dots ,x_{\theta}$ and braiding $c(x_i\otimes x_j)=q_{ij}(x_j\otimes x_i)$? If so,
compute $\dim_\Bbbk \cB(V)$, and give a ''nice'' presentation by generators and relations.
\end{question}
\noindent
Several authors obtained the classification result for infinite and finite dimensional Nichols algebra of Cartan type, see~\cite{a-AndrSchn00,a-Heck06a,Rosso98}.
I.~Heckenberger determined all finite dimensional Nichols algebra of diagonal type~\cite{a-Heck04d,a-Heck04bb,a-Heck09}, that is when $V$ is the direct sum of 1-dimensional \YD modules over $H_{0}$.
The generators and relations of such Nichols algebras were also given~\cite{Ang1,Ang2}.

The methods developed in the study of the generalizations of Lie algebras are useful to analyze Nichols algebras~\cite{b-Bahturin92}. V.~Kharchenko proved that any Hopf algebra generated by skew-primitive and group-like elements has a restricted \PBW basis~\cite[Theosrem~2]{a-Khar99}. Note that V.~Kharchenko results can apply to Nichols algebras of diagonal type. Motivated by the close relation to Lie theory, I.~Heckenberger~\cite{a-Heck06a} defined the arithmetic root system and Weyl groupoid for Nichols algebras $\cB(V)$ of diagonal type.
Late, M.~Cuntz, I.~Heckenberger and H.~Yamane developed the combinatorial theory of
these two structures~\cite{c-Heck09b,Y-Heck08a}.
Then the theory of root systems and Weyl groupoids was carried out in more general Nichols algebras~\cite{a-AHS08,HS10,a-HeckSchn12a}.
Further, all finite Weyl groupoids were classified in~\cite{c-Heck12a,c-Heck14a}.
Those constructions are very important theoretical tools for the classification of Nichols algebra $\cB(V)$.

With the classification result, N.~Andruskiewitsch and H.-J.~Schneider~\cite{a-AndrSchn05} obtained a classification theorem about finite-dimensional pointed Hopf algebras under some technical assumptions. Based on such successful applications,
the analyze to Nichols algebras over arbitrary fields is crucial and has also potential applications.
Towards this direction, new examples of Nichols algebras in positive characteristic and a combinatorial formula to study the relations in Nichols algebras were found~\cite{clw}.
Over fields of positive characteristic,
rank 2 and rank 3 finite dimensional Nichols algebras of diagonal type were listed in~\cite{WH-14,W-17}.
In this paper, we give the complete classification result of rank 4 case.
Besides, the notations and conventions in~\cite{WH-14,W-17} are followed and several results from these papers will be used.

The paper is organized as follows.
Section~\ref{se:Pre} is devoted to preliminaries.
In Section~\ref{se:Rank4Cartan}, we explicitly characterize finite connected indecomposable Cartan graphs of rank 4.
In order to do that,
we introduce good $A_4$ neighborhood and good $B_4$ neighborhood
see Definitions~\ref{defA4} and ~\ref{defB4}.
In Theorem~\ref{Theo:goodnei},
we prove that every finite connected indecomposable Cartan graph of rank 4 contains a point which has at least one of the good neighborhoods. Theorem~\ref{Theo:goodnei} allows us to avoid complicated computations in the final proof of Theorem~\ref{theo:clasi}.
Finally, in Section~\ref{se:clasi} we formulate the classification Theorem~\ref{theo:clasi} and present all the possible generalized Dynkin diagrams of rank 4 braided vector spaces of diagonal type with a finite root system over arbitrary fields in Table~(\ref{tab.1}).
As a corollary of Theorem~\ref{theo:clasi},
all rank 4 finite-dimensional Nichols algebras of diagonal type in positive characteristic are given,
see Corollary~\ref{coro-cla}.

\section{Nichols algebras of diagonal type}\label{se:Pre}
In this section, we recall \YD modules, braided vector spaces, and their relations.
The main object of this paper is also presented.
For further details on these topics we refer to \cite{inp-Andr14,inp-Andr02,a-AndrGr99}

\subsection{\YD modules}
Let $\Bbbk$ be a field of characteristic $p> 0$.
Let $\Bbbk^*=\Bbbk\setminus \{0\}$, $\ndN_0=\ndN\bigcup \{0\}$, $\theta\in \ndN_0$, and $I=\{1,\dots,\theta\}$,.
We start by recalling the main objects.
\begin{defn}
Let $V$ a $\theta$-dimensional vector space over $\Bbbk$.
The pair $(V, c)$ is called a \textit{braided vector space}, if
$c\in \Aut(V\otimes V)$ is a solution of the braid equation,
that is
\begin{equation*}
 (c \otimes \id)(\id \otimes c)(c \otimes \id) = (\id \otimes c)(c \otimes \id)(\id\otimes c).
\end{equation*}
A braided vector space $(V, c)$ is termed \textit{of diagonal type}
if $V$ admits a basis $\{x_i | i\in I\}$ such that for all $i, j \in I$ one has
\begin{equation*}
  c(x_i \otimes x_j) = q_{ij}x_j \otimes x_i \quad \textit{for some} \quad q_{ij} \in \Bbbk^*.
\end{equation*}
\end{defn}
The matrix $(q_{ij})_{i,j\in I}$ is termed of the \textit{braiding matrix} of $V$. We say that the braiding matrix $(q_{ij})_{i,j\in I}$ is \textit{indecomposable} if for any $i\not=j$ there exists a sequence $i_1 = i, i_2, \dots, i_t = j$ of elements of $I$ such that $q_{i_si_{s+1}}q_{i_{s+1}i_s}\not= 1$, where $1\leq s\leq t-1$. In this paper, we mainly concern the braided vector spaces of diagonal type with an indecomposable braiding matrix.

\begin{defn}
Let $H$ be a Hopf algebra.
A \textit{\YD module} $V$ over $H$ is a left $H$-module with left action $.$ : $H\otimes  V \longrightarrow V$ and a left $H$-comodule
with left coaction $\delta_L : V \longrightarrow H \otimes V$ satisfying the compatibility condition
\begin{equation*}
\delta_L(h.v) = h_{(1)}v_{(-1)} \kappa (h_{(3)}) \otimes h_{(2)}.v_{(0)}, h \in H, v\in V.
\end{equation*}
We say that $V$ is \textit{of diagonal type} if $H=\Bbbk G$ and $V$ is a direct sum of one-dimensional \YD modules over the group algebra $\Bbbk G$, where $G$ is abelian.
\end{defn}
\indent
We denote by $\ydH$ the category of \YD modules over $H$, where morphisms preserve both the action and the coaction of $H$.
The category $\ydH$ is braided with braiding
\begin{equation}\label{def.brading}
  c_{V,W}(v\otimes w)=v_{(-1)}.w \otimes v_{(0)}
\end{equation}
for all $V, W\in \ydH$, $v\in V$, and $w\in W$.
Actually, the category $\ydH$ ia a braided monoidal category, where the monoidal structure is given by the tensor product over $\Bbbk$. Then any \YD module $V\in \ydH$ over $H$ admits a braiding $c_{V,V}$ and hence $(V, c_{V,V})$ is a braided vector space. Conversely, any braided vector space can be realized as a \YD module over some Hopf algebras if and only if the braiding is rigid \cite[Section 2]{T00}. Notice that \YD module structures on $V$ with different Hopf algebras can give the same braiding and not all braidings of $V$ are induced by the above Equation~(\ref{def.brading}).

If $H=\Bbbk G$ then we write $\ydD$ for the category of \YD modules over $\Bbbk G$ and say that $V\in \ydD$ is a \YD module over $G$. Notice that if $V\in \ydD$ is of diagonal type then $(V, c_{V,V})$ is a braided vector space of diagonal type. Any braided vector space of diagonal type is also a \YD module of diagonal type. Indeed, assume that $(V, c)$ is a braided vector space of diagonal type with an indecomposable braiding matrix $(q_{ij})_{i, j\in I}$ of a basis $\{x_i|i\in I\}$. Let $G_0$ be an abelian group generated by elements $\{g_i|i\in I\}$. Define the left coaction and left action by
\[
\delta_L(x_i)=g_i\ot x_i \in G_0\otimes V, \quad g_i\lact x_j=q_{ij}x_j\in V.
\]
 Then $V=\oplus_{i\in I}\Bbbk x_i$ and each $\Bbbk x_i$ is one-dimensional \YD modules over $G_0$. Hence $V$ is a \YD module of diagonal type over $G_0$.

\subsection{Nichols algebras of diagonal type}
Let $(V, c)$ be a $\theta$-dimensional braided vector space of diagonal type. In this section, we recall a definition of the Nichols algebra $\cB(V)$ generated by $(V, c)$.
In order to do that, we introduce one more notion in the category $\ydH$.

\begin{defn}
  Let $H$ be a Hopf algebra.
  A \textit{braided Hopf algebra} in $\ydH$ is a $6$-tuple $\cB=(\cB, \mu, 1, \delta, \epsilon, \kappa_B)$,
  where $(B, \mu, 1)$ is an algebra in $\ydH$ and a coalgebra $(B, \delta, \epsilon)$ in $\ydH$, and $\kappa_B: B\rightarrow B$ is a morphism in $\ydH$ such that $\delta$, $\epsilon$ and $\kappa_B$ satisfy $ \kappa_B(b^{(1)})b^{(2)}=b^{(1)}\kappa_B(b^{(2)})=\epsilon(b) 1$, where we define $\delta(b)=b^{(1)}\otimes b^{(2)}$ as the coproduct of $\cB$ to avoid the confusion.
\end{defn}

\begin{defn}
  The \textit{Nichols algebra} generated by $V\in \ydH$ is defined as the quotient
  \[
  \cB(V)=T(V)/\cI(V)=(\oplus_{n=0}^{\infty} V^{\otimes n})/\cI(V)
  \]
  where $\cI(V)$ is the unique maximal coideal among all coideals of $T(V)$ which are contained in $\oplus_{n\geq 2}V^{\otimes n}$.
  Nichols algebra $\cB(V)$ is said to be~\textit{of diagonal type} if $V$ is a \YD module of diagonal type.
  The dimension of $V$ is the \textit{rank} of Nichols algebra $\cB(V)$.
\end{defn}
Note that if $V\in \ydH$ and $V$ is an algebra in $\ydH$ then $B\otimes B$ is an algebra in $\ydH$ with the product given by
\begin{equation}\label{eq-alg}
  (a\otimes b)(c\otimes d)=a(b_{(-1)}\lact c)\otimes b_{(0)}d,
\end{equation}
for all $a, b, c, d\in V$, where $.$ denotes the left action of $H$ on $V$.

The \textit{tensor algebra} $T(V)$ admits a natural structure of a \YD module and an algebra structure in $\ydH$. It is then a braided Hopf algebra in $\ydH$ with coproduct $\delta(v)=1\otimes v+ v\otimes 1\in T(V)\otimes T(V)$ and counit $\epsilon(v)=0$ for all $v\in V$ such that $\delta$ and $\epsilon$ are the algebra morphisms. The antipode of $T(V)$ exists, see \cite[Section 2.1]{inp-AndrSchn02}. Notice that the product defined by Equation (\ref{eq-alg}) on $T(V)$ is the unique algebra structure such that $\delta(v)=1\otimes v+ v\otimes 1\in T(V)\otimes T(V)$ for all $v\in V$. The coproduct can be extended from $V$ to $T(V)$. For example, for all $v, w\in V$ one gets (we write the elements of $T(V)$ without the tensor product sign for brevity)
 \begin{align*}
 \begin{split}
 \delta(vw)=&\delta(v)\delta(w)\\
           =&(1\otimes v+ v\otimes 1)(1\otimes w+ w\otimes 1)\\
           =& 1\otimes vw+ v_{(-1)}.w\otimes v_{(0)}+v\otimes w+vw\otimes 1.
\end{split}
\end{align*}
Let $(I_i)_{i\in I}$ be the family of all ideals of $T(V)$ contained in $\oplus_{n\geq 2}V^{\otimes n}$, i.e.
\[\delta(I_i)\subset I_i\otimes T(V)+ T(V)\otimes I_i.\]
 Then the ideal $\cI(V):=\sum_{i\in I}I_i$ is the largest element in $(I_i)_{i\in I}$.
Hence $\cB(V)$ is a braided Hopf algebra in $\ydH$. As proved in \cite[Propssition 3.2.12]{a-AndrGr99}, Nichols algebra $\cB(V)$ is the unique $\ndN_0$-graded braided Hopf algebra generated by $V$ in $\ydH$ with homogenous components $\cB(V)(0)=\Bbbk$, $\cB(V)(1)=V$, and $P(\cB(V))=V$, where $P(\cB(V))$ is the space of primitive elements of $\cB(V)$.

\subsection{Weyl groupoids}
In this section, we recall the notations of semi-Cartan graphs, root systems and Weyl groupoids. We mainly follow the terminology from \cite{c-Heck09a},\cite{Y-Heck08a}. See also~\cite{WH-14} and~\cite{W-17}.
\begin{defn} \it
A \textit{generalized Cartan matrix} is a matrix $A =(a_{ij})_{i,j\in I}$ with integer entries such that
\begin{itemize}
\itemsep=0pt
\item $a_{ii}=2$ and $a_{jk}\le 0$ for any $i,j,k\in I$ with
    $j\not=k$,
\item  if $a _{ij}=0$ for some $i,j\in I$, then $a_{ji}=0$.
\end{itemize}
\end{defn}
\noindent A generalized Cartan matrix $A\in \ndZ ^{I \times I}$ is \textit{decomposable} if there exists a nonempty proper subset $I_1\subset I$ such that $a_{ij}=0$ for any $i \in I_1$ and $j \in I\setminus I_1$. We say that $A$ is  \textit{indecomposable} if $A$ is not decomposable.

\begin{defn}
  Let $\cX$ be a non-empty set and $A^X =(a_{ij}^X)_{i,j \in I}$ a generalized Cartan matrix for all $X\in \cX$.
  For any $i\in I$ let $r_i : \cX \to \cX$, $X \mapsto r(i,X)$, where $r: I\times \cX \to \cX$ is a map.
  The quadruple
  \[\cC = \cC (I, \cX, r, (A^X)_{X \in \cX})
  \]
  is called a \textit{semi-Cartan graph} if
  $r_i^2 = \id_{\cX}$ for all $i \in I$,  and $a^X_{ij} = a^{r_i(X)}_{ij}$ for all $X\in \cX$ and $i,j\in I$.
  We say that a semi-Cartan graph $\cC$ is~\textit{indecomposable}
  if $A^X$ is indecomposable for all $X\in \cX$.
\end{defn}
\noindent For the sake of simplicity,
the elements of the set $\{r_i(X), i\in I\}$ are termed the \textit{neighbors} of $X$ for all $X\in \cX$.
The cardinality of $I$ is termed the \textit{rank} of $\cC$ and the elements of $\cX$ are the \textit{points} of $\cC$.

\begin{defn}
The \textit{exchange graph} of $\cC$ is a labeled non-oriented graph with vertices set $\cX$
and edges set $I$, where two vertices $X, Y$ are connected by an edge $i$ if and only if $X\not=Y$ and $r_i(X)=Y$ (and $r_i(Y)=X$).
We display one edge with several labels instead of several edges for simplification. A semi-Cartan graph $\cC$ is said to be \textit{connected} if its exchange graph is connected.
\end{defn}
\noindent

For the remaining part of this section, we assume that $\cC = \cC (I, \cX, r, (A^X)_{X \in \cX})$ is a connected semi-Cartan graph. 
Let $(\alpha_i)_{i\in I}$ be the standard basis of $\ndZ^I$. For all $X\in \cX$ let
\begin{equation*}
s_i^X\in \Aut(\ndZ^I), ~~
s_i^X \alpha _j=\alpha_j-a_{ij}^X \alpha_i.
\end{equation*}
for all $j\in I$.
Let $\cD(\cX, I)$ be the category such that Ob$\cD(\cX, I)=\cX$ and morphisms $\Hom(X,Y)=\{(Y,f,X)|f\in \End(\ndZ ^I)\}$ for $X, Y\in \cX$ with the composition $(Z,g,Y)\circ (Y,f,X)=(Z, gf, X)$ for all $ X, Y, Z\in \cX$, $f,g\in \End(\ndZ ^I)$. Let \textit{$\cW(\cC)$} be the smallest subcategory of $\cD(\cX, I)$, where the morphisms are generated by $(r_i(X), s_i^X, X)$, with $i\in I$, $X\in \cX$. From now on, we write $s_i^X$ instead of $(r_i(X), s_i^X, X)$, if no confusion is possible. Notice that all generators $s_i^X$ are reflections and hence are invertible. Then $\cW(\cC)$ is a groupoid.

For any category $\cD$ and any object $X$ in $\cD$ let $\Hom(\cD, X)=\cup_{Y\in \cD}\Hom(Y,X)$.
\begin{defn}
For all $X\in \cX$, the set
\begin{equation}\label{eq-realroots}
  \rersys{X}=\{\omega\alpha_i \in \ndZ^I|\omega \in \Hom(\cW(\cC),X)\}
\end{equation}
 is called the set of real roots of $\cC$ at $X$. The elements of $\rersys{X}_{\boldsymbol{+}}=\rersys{X}\cap \ndN_0^I$ are called positive roots and those of $\rersys{X}\cap -\ndN_0^I$ negative roots, denoted by $\rersys{X}_{\boldsymbol{-}}$.
If the set $\rersys{X}$ is finite for all $X\in \cX$ then we say that $\cC$ is finite.
\end{defn}

\begin{defn}
We say that $\cR=\cR(\cC,(\rsys ^{X})_{X\in \cX})$ is a \textit{root system of type $\cC$} if for all $X \in \cX $, the sets $\rsys ^X $ are the subsets of $\ndZ ^I$ such that
\begin{itemize}
\itemsep=0pt
\item $\rsys ^X=(\rsys^ X\cap \ndN _0^I)\cup -(\rsys ^X\cap \ndN _0^I)$.
\item $\rsys ^X\cap \ndZ \al _i=\{\al _i,-\al _i\}$ for all $i\in I$.
\item $s_i^X(\rsys ^X)= \rsys ^{r_i(X)}$ for all $i \in I$.
\item $(r_i r_j)^{m_{ij}^X}(X) = X$ for any $i,j \in I$ with $i\not=j$ where $m_{ij}^X= |\rsys^X\cap (\ndN _0\al _i + \ndN _0 \al_j)|$ is finite.
\end{itemize}
\end{defn}
\noindent We say that $\cW(R):=\cW(\cC)$ is the groupoid of $\cR$. As in \cite[Definition 4.3]{c-Heck09b} we say that $\cR$ is \textit{reducible} if there exist non-empty disjoint subsets of $I',I''\subset I$ such that $I=I'\cup I''$ and $a_{i j}=0$ for all $i\in I'$, $j\in I''$ and
  \[
   \rsys ^{X}=\Big(\rsys ^{X}\cap \sum _{i\in I'}\ndZ \al _i\Big)\cup
  \Big(\rsys ^{X}\cap \sum _{j\in I''}\ndZ \al _j\Big)\qquad
  \text{for all}\quad X\in \cX.
  \]
In this case, we write $\cR =\cR |_{I_1}\oplus \cR |_{I_2}$. If $\cR \not=\cR |_{I_1}\oplus \cR |_{I_2}$ for all non-empty disjoint subsets $I_1,I_2\subset I$, then $\cR$ is termed \textit{irreducible}.
\begin{defn}
   Let $\cR=\cR(\cC,(\rsys ^{X})_{X\in \cX})$ be a root system of type $\cC$. We say that $\cR$ is~\textit{finite} if $\rsys ^{X}$ is finite for all $X\in \cX$.
\end{defn}

Let $\cR=\cR(\cC,(\rsys ^{X})_{X\in \cX})$ be a root system of type $\cC$. We recall some properties of $\cR$ from \cite{c-Heck09b} and \cite{c-Heck12a}.
\begin{lemma}\label{lem:jik}
 Let $X\in \cX$, $k \in \ndZ$, and $i, j\in I$ such that $i\not= j$. Then $\alpha _j + k\alpha_i\in \rersys{X}$ if and only if $0\leq k\leq -a_{ij}^X$.
\end{lemma}
 Notice that the finiteness of $\cR$ does not mean that $\cW(\cR)$ is also finite, since the set $\cX$ might be infinite.
\begin{lemma}\label{lem:finite}
   Let $\cC = \cC (I, \cX, r, (A^X)_{X \in \cX})$ be a connected semi-Cartan graph and $\cR=\cR(\cC,(\rsys ^{X})_{X\in \cX})$ is a root system of type $\cC$. Then the following are equivalent.
   \begin{itemize}
   \itemsep=0pt
     \item[$(1)$] $\cR$ is finite.
     \item[$(2)$] $\rsys^{X}$ is finite for some $X\in \cX$.
     \item[$(3)$] $\cC$ is finite.
     \item[$(4)$] $\cW(\cR)$ is finite.
   \end{itemize}
\end{lemma}

Recall that $\cC$ is a connected semi-Cartan graph. Then we get the following.
\begin{prop}\label{prop.indecom}
  Let $\cR=\cR(\cC,(\rsys ^{X})_{X\in \cX})$ be a root system of type $\cC$. Then the following are equivalent.
  \begin{itemize}
  \itemsep=0pt
    \item[$(1)$] There exists $X\in \cX$ such that $A^X$ is indecomposable
    \item[$(2)$] The semi-Cartan graph $\cC$ is indecomposable.
   \end{itemize}
    If $\cR$ is finite then the semi-Cartan graph $\cC$ is indecomposable if and only if the root system $\cR$ is irreducible.
\end{prop}

\begin{defn}\label{def.Weylgroupoid}
We say $\cC$ is a Cartan graph if the following hold:
\begin{itemize}
\itemsep=0pt
  \item For all $X\in \cX$ the set $\rersys{X}=\rersys{X}_{\boldsymbol{+}}\bigcup \rersys{X}_{\boldsymbol{-}}$.
  \item If $l_{mn}^Y:=|\rersys{Y}\cap (\ndN_0 \alpha_m+\ndN_0 \alpha_n)|$ is finite, then $(r_m r_n)^{l_{mn}^Y}(Y)=Y$, where $m, n\in I$, $Y\in \cX$.
\end{itemize}
In this case, $\cW(\cC)$ is called the Weyl groupoid of $\cC$.
\end{defn}

Let $\cR^{re}:=\cR(\cC,(\rersys{X})_{X\in \cX})$. Then
 $\cC$ is a Cartan graph if and only if $\cR^{re}$ is a root system of type $\cC$. Indeed, we get that $s_i^X(\rersys{X})=\rersys{r_i(X)}$ by Equation (\ref{eq-realroots}). For all $X\in \cX$, we obtain that $\rersys{X}=\rersys{X}_{\boldsymbol{+}}\cup \rersys{X}_{\boldsymbol{-}}$, since $\omega s_i^{r_i(X)}(\al_i)=-\omega(\al_i)$ for any $\omega \in \Hom(\cW(\cC),X)$.

The following proposition implies that if $\cR$ is a finite root system of type $\cC$ then $\cR=\cR^{re}$, namely, all roots are real and $\cR$ is uniquely determined by $\cC$.
 \begin{prop} \label{prop.allposroots}
  Let $\cR=\cR(\cC,(\rsys ^{X})_{X\in \cX})$ be a root system of type $\cC$. Let $X\in \cX$, $m\in \ndN _0$, and $i_1,\ldots ,i_m\in I$ such that
  \[\omega =\id_X s _{i_1} s_{i_2}\cdots s_{i_m}\in \Hom(\cW(\cC), X)\] and $\ell (\omega )=m$.
  Then the elements
  \[
    \beta _n=\id_X s_{i_1} s_{i_2} \cdots s_{i_{n-1}}(\alpha_{i_n})\in \rsys^ X\cap \ndN _0^I,
  \]
  are pairwise different, where $n\in \{1,2,\ldots ,m\}$ (and $\beta _1=\alpha _{i_1}$). Here,
  \[\ell(\omega)=\mathrm{min}\{m\in \ndN_0|\omega=\id_X s_{i_1}s_{i_2}\cdots s_{i_m}, i_1, i_2, \ldots, i_m\in I\}\]
   is the length of $\omega\in \Hom(\cW(\cC), X)$.
  In particular, if $\cR$ is finite and
  $\omega \in \Hom (\cW (\cC ))$ is the longest element,
  then
  \[
    \{\beta _n\,|\,1\le n\le \ell (\omega )=|\rsys ^{X}|/2\}=\rsys^ X\cap \ndN _0^I.
  \]
\end{prop}

\begin{rem}
 If $\cC$ is a finite Cartan graph then $\cR$ is finite and hence $$\cR^{re}=\cR^{re}(\cC,(\rersys{X})_{X\in \cX})$$ is the unique root system of type $\cC$ by Proposition~\ref{prop.allposroots}, that is, $\cR$ is uniquely determined by $\cC$.
\end{rem}
\subsection{Cartan graphs for Nichols algebras of diagonal type}
In this section we attach a semi-Cartan graph to a tuple of finite-dimensional \YD modules under some finiteness conditions.
Let $G$ be an abelian group. Let $\ffg$ be the set of $\theta$-tuples of finite-dimensional irreducible objects in $\ydD$ and $\fiso^G$ be the set of $\theta $-tuples of isomorphism classes of finite-dimensional irreducible objects in $\ydD$. For any $(M_1, \dots, M_{\theta})\in \ffg$, write $[M]:=([M_1], \dots, [M_{\theta}])\in \fiso^G$ the corresponding isomorphism class of $(M_1, \dots, M_{\theta})$.

Assume that $V_M=\oplus_{i\in I}\Bbbk x_i\in \ydD$ is a \YD module of diagonal type over $G$, where $\{x_i|i\in I\}$ is a basis of $V$. Then there exists a matrix $(q_{ij})_{i,j\in I}$ such that $\delta(x_i)=g_i\otimes x_i$ and $g_i. x_j=q_{ij}x_j$ for all $i,j\in I$. We fix that $M=(\Bbbk x_1, \Bbbk x_2, \dots, \Bbbk x_{\theta})\in \ffg$ is a tuple of one-dimensional \YD over $G$ and $[M]\in \fiso^G$. We say that the matrix $(q_{ij})_{i,j\in I}$ is the \textit{braiding matrix of $M$}. Recall that the matrix is independent of the basis $\{x_i|i\in I\}$ up to permutation of $I$. We say $\cB(V_M)=\cB(\oplus_{i=1}^{n}\Bbbk x_i)$ is the Nichols algebra of the tuple $M$, denoted by $\cB(M)$.

Recall that the adjoint representation $\ad$ of a Nichols algebra $\cB(V)$ from \cite{a-AndrSchn98} is the linear map $\ad_c:$ $V\rightarrow \End(\cB(V))$
\[
\ad_{c} x(y)=\mu (\id-c)(x\otimes y)=xy-(x_{(-1)}\lact y)x_{(0)}
\]
for all $x\in V$, $y\in \cB(V)$, where $\mu$ is the multiplication map of $\cB(V)$ and $c$ is defined by Equation (\ref{def.brading}).
In particular, the braided commutator $ad_{c}$ of $\cB(M)$ takes the form
\[ad_{c}x_i(y)=x_i y-(g_i\lact y) x_i~ \textit{for all} ~i\in I,~ y\in \cB(M).\]

In order to construct a semi-Cartan graph to $M$, we recall some finiteness conditions from \cite{inp-AndrSchn02} and \cite{HS10}.
\begin{defn}
Let $i\in I$. We say that $M$ is \textit{$i$-finite}, if for any $j\in I\setminus \{i\}$, $(\ad_{c} x_i)^m (x_j)=0$ for some $m\in \ndN$.
\end{defn}
\begin{lemma}\label{le:aijM}
  For any $i,j\in I$ with $i\not=j$,s
  the following are equivalent.
  \begin{itemize}
    \item $(m+1)_{q_{ii}}(q_{ii}^mq_{ij}q_{ji}-1)=0$ and $(k+1)_{q_{ii}}(q_{ii}^kq_{ij}q_{ji}-1)\not=0$ for all $0\leq k<m$.
    \item $(ad_{c}x_i)^{m+1}(x_j)=0$ and $(ad_{c}x_i)^m(x_j)\not=0$ in $\cB (V)$.
  \end{itemize}
  Here $(n)_q:=1+q+\cdots+q^{n-1}$,
  which is $0$ if and only if $q^n=1$ for $q\not=1$ or $p|n$ for $q=1$. Notice that $(1)_q\not=0$ for any $q\in \ndN$.

\end{lemma}
Hence we get the following from Lemma~\ref{le:aijM}.
\begin{lemma}\label{lem:aij}
  Let $i\in I$.
  Then $M=(\Bbbk x_j)_{j\in I}$ is $i$-finite if and only if for any $j\in I\setminus\{i\}$ there is a non-negative integer $m$ satisfying $(m+1)_{q_{ii}}(q_{ii}^mq_{ij}q_{ji}-1)=0$.
\end{lemma}

Let $i\in I$. Assume that $M$ is $i$-finite. Let $(a_{ij}^{M})_{j\in I}\in \ndZ^I$ and $R_i(M)=({R_i(M)}_j)_{j\in I}$, where
\begin{align*}
  a_{ij}^M=&
  \begin{cases}
    2& \text{if $j=i$,}\\
    -\mathrm{max}\{m\in \ndN_0 \,|\,(\ad_c  x_i)^m(x_j)\not=0 \}& \text{if $j\not=i$.}
  \end{cases}
\end{align*}
\begin{equation}\label{eq-ri}
{R_i(M)}_i= \Bbbk y_i,\qquad
{R_i(M)}_j= \Bbbk(\ad_{c}x_i)^{-a_{ij}^M}(x_j),
\end{equation}
where $y_i\in (\Bbbk x_i)^*\setminus \{0\}$. If $M$ is not $i$-finite, then let $R_i(M)=M$.
Then $R_i(M)$ is a $\theta$-tuple of one-dimensional \YD modules over $G$.

Let
\[
\ffg(M)=\{R_{i_1} \cdots R_{i_n}(M)\in \ffg|\, n\in \ndN_0, i_1,\dots, i_n\in I\}
\]
and
\[
\fiso^G(M)=\{[R_{i_1} \cdots R_{i_n}(M)]\in \fiso^G |\,n\in \ndN_0, i_1,\dots, i_n\in I\}.
\]
\begin{defn}\label{defn-admitsallref}
  We say that $M$ \textit{admits all reflections} if $N$ is $i$-finite for all $N\in \ffg(M)$.
\end{defn}
Notice that the reflections depend only on the braiding matrix $(q_{ij})_{i,j\in I}$. We recall the notion of generalized Dynkin diagram for a braided vector spaces of diagonal type~\cite{a-Heck04e}.
\begin{defn} 
Let $V$ be a $\theta$-dimensional braided vector space of diagonal type with the braiding matrix $(q_{ij})_{i,j\in I}$. The $\textit{generalized Dynkin diagram}$ of $V$ is a non-directed graph $\cD$ with the following properties:
\begin{itemize}
\itemsep=0pt
\item there is a bijective map $\phi$ from $I$ to the vertices of $\cD$,
\item for all $i\in I$ the vertex $\phi (i)$ is labeled by $q_{ii}$,
\item  for all $i,j\in I$ with $i\not=j$,
 the number $n_{ij}$ of edges between $\phi (i)$ and $\phi (j)$ is either $0$ or $1$. If $q_{ij}q_{ji}=1$ then $n_{ij}=0$, otherwise $n_{ij}=1$ and the edge is labeled by $q_{ij}q_{ji}.$
\end{itemize}
\end{defn}
We say that the \textit{generalized Dynkin diagram of $M$} is the generalized Dynkin diagram of braided vector space $\oplus_{i\in I}M_i$.
Notice that the generalized Dynkin diagram of $M$ is connected if the braiding matrix of $M$ is indecomposable.

In more details, one can obtain the labels of the generalized Dynkin diagram of $R_i(M)=(R_i(M)_j)_{j\in I}$ by the following lemma.
\begin{lemma}\label{le:Dynkin}
Let $i\in I$.
Assume that $M$ is $i$-finite and let $a_{ij}:=a_{ij}^M$ for all $j\in I$.
Let $(q'_{jk})_{j,k\in I}$ be the braiding matrix of $R_i(M)$ with respect to $(y_j)_{j\in I}$.
Then
  \begin{gather*}
q_{jj}'=
 \begin{cases}
 q_{ii} & \text{if $j=i$},\\
 q_{jj}  & \text{if $j\not=i$, $q_{ij}q_{ji}=q_{ii}^{a_{ij}}$},\\
 q_{ii}q_{jj}{(q_{ij}q_{ji})^{-a_{ij}}} & \text{if $j\not=i$, $q_{ii}\in G'_{1-a_{ij}}$},\\
 q_{jj}{(q_{ij}q_{ji})^{-a_{ij}}} & \text{if $j\not=i$, $q_{ii}=1$},
 \end{cases}
 \end{gather*}
 \begin{gather*}
 q_{ij}'q_{ji}'=
 \begin{cases}
 q_{ij}q_{ji} & \text{if $j\not=i$, $q_{ij}q_{ji}=q_{ii}^{a_{ij}}$},\\
 q_{ii}^2(q_{ij}q_{ji})^{-1} & \text{if $j\not=i$, $q_{ii}\in G'_{1-a_{ij}}$},\\
 (q_{ij}q_{ji})^{-1} & \text{if $j\not=i$, $q_{ii}=1$},
 \end{cases}
 \end{gather*}
 and
 \begin{gather*}
q_{jk}'q_{kj}'=
 \begin{cases}
 q_{jk}q_{kj}  & \text{if $q_{ir}q_{ri}=q_{ii}^{a_{ir}}$, $r\in \{j, k\}$},\\
 q_{jk}q_{kj}(q_{ik}q_{ki}q_{ii}^{-1})^{-a_{ij}}& \text{if $q_{ij}q_{ji}=q_{ii}^{a_{ij}}$, $q_{ii}\in G'_{1-a_{ik}}$},\\
 q_{jk}q_{kj}(q_{ij}q_{ji})^{-a_{ik}}(q_{ik}q_{ki})^{-a_{ij}} &  \text{if $q_{ii}=1$,}\\
 q_{jk}q_{kj}q_{ii}^{2}(q_{ij}q_{ji}q_{ik}q_{ki})^{-a_{ij}} & \text{if $q_{ii}\in G'_{1-a_{ik}}$, $q_{ii}\in G'_{1-a_{ij}}$}.
  \end{cases}
 \end{gather*}
 for $j, k\not=i$, $j\not=k$.
Here, $G_n'$ denotes the set of primitive $n$-th roots of unity
in $\Bbbk$, that is $G'_n=\{q\in \Bbbk^*|\,\, q^n=1, q^k\not=1~\text{for all}~ 1\leq k < n\}$ for $n\in \ndN$.
 \end{lemma}

If $M$ admits all reflections, then we are able to construct a semi-Cartan graph $\cC(M)$ of $M$ by the following theorem.
\begin{theorem}\label{theo.regualrcar}
  Assume that $M$ admits all reflections.
  For all $X\in \fiso^G(M)$
  let \[[X]_{\theta}=\{Y\in \fiso^G(M)| \,\text{Y and X have the same generalized Dynkin diagram}\}.\]
  Let $\cY_{\theta}(M)=\{[X]_{\theta} |\, X\in \fiso^G(M)\}$
  and $A^{[X]_{\theta}}=A^X$ for all $X\in \fiso^G(M)$.
  Let $t: I\times \cY_{\theta}(M)\rightarrow \cY_{\theta}(M)$,
  $(i, [X]_{\theta})\mapsto [R_i(X)]_{\theta}$.
  Then the tuple
  \[
  \cC(M)=\{I, \cY_{\theta}(M), t, (A^Y)_{Y\in \cY_{\theta}(M)}\}
  \]
  is a connected semi-Cartan graph. We say that $\cC(M)$ is the semi-Cartan graph attached to $M$.
\end{theorem}
\begin{proof}
Let $X\in \fiso^G(M)$. Since $M$ admits all reflections, we obtain that all entries of $A^X$ are finite. Hence $A^X$ is well-defined.  Moreover, if $a_{ij}^X=0$ then $a_{ji}^X=0$ by Lemma \ref{lem:aij}. Hence $A^X$ is a well-defined generalized Cartan matrix for all $X\in \fiso^G(M)$. For any $X, Y\in \fiso^G(M)$, if $X$ and $Y$ have the same generalized Dynkin diagram then $A^X=A^Y$ and hence $A^{[X]_{\theta}}=A^{[Y]_{\theta}}$. Then $A^{[X]_{\theta}}$ is well-defined for all $X\in \fiso^G(M)$. Hence $\{A^Y\}_{Y\in \cY_{\theta}(M)}$ is a family of generalized Cartan matrices. Besides, if $N$ is $i$-finite then $a_{ij}^N=a_{ij}^{R_i(N)}$ and $R^2_i(N)=N$ for all $N\in \fiso^G(M)$ by \cite[Theorem 3.12(2)]{a-AHS08}. Hence $t_i$ is a reflection map for all $i\in I$. Then $\cC(M)$ is a well-defined semi-Cartan graph. From the construction of the reflection $R_i$ by Equation (\ref{eq-ri}) we obtain that $\cC(M)$ is connected.
\end{proof}
Furthermore, one can attach a groupoid $\cW(M):=\cW(\cC(M))$ to $M$ if $M$ admits all reflections.

Notice that
Nichols algebra $\cB(M)$ is $\ndN_0^{\theta}$-graded with
$\deg M_i=\al_i $ for all $i\in I$.
Following the terminology in~\cite{HS10},
we say that the
Nichols algebra $\cB(M)$ is \textit{decomposable}
if there exists a totally ordered index set $(L,\le)$ and a sequence
$(W_l)_{l\in L}$ of finite-dimensional irreducible $\ndN _0^\theta $-graded objects in
$\ydD $
such that
\begin{equation}\label{eq-decom}
  \cB (M)\simeq
  \bigotimes _{l\in L}\cB (W_l).
\end{equation}
For each decomposition~(\ref{eq-decom}),
we define the set of~ \textit{positive roots} $\rsys^{[M]}_{+}\subset \ndZ^I$ and the set of~ \textit{roots} $\rsys^{[M]}\subset \ndZ^I$ of $[M]$ by
$$\rsys^{[M]}_{+}=\{\deg(W_l)|\, l\in L\}, \quad
\rsys^{[M]}=\rsys^{[M]}_{+}\cup-\rsys^{[M]}_{+}.$$
By~\cite[Theorem~4.5]{HS10} we obtain that
the set of roots $\rsys^{[M]}$ of $[M]$ does not depend on the choice of the decomposition.

\begin{rem}\label{rem-decom}
 If $\dim M_i=1$ for all $i\in I$,
 then the Nichols algebra $\cB(M)$ is decomposable based on the theorem of V.~Kharchenko~\cite[Theorem~2]{a-Khar99}.The set of roots of Nichols algebra $\cB(M)$ can be always defined and it is denoted by $\rsys^{[M]}$.
 If the set of roots $\rsys^{[M]}$ is finite then we can check that $M$ admits all reflections by \cite[Corollary 6.12]{HS10}.
\end{rem}
If $M$ admits all reflections and $\rsys^{[M]}$ is finite, then we can define a finite root system $\cR(M)(\cC(M),(\rsys^{[N]})_{N\in \ffg(M)})$ of type $\cC(M)$.
\begin{theorem} \label{thm:rootofR_M}
 Assume that $M$ admits all reflections. Then the following are equivalent.
\begin{itemize}
\itemsep=0pt
\item[$(1)$] $\rsys^{[M]}$ is finite.
\item[$(2)$] $\cC(M)$ is a finite Cartan graph.
\item[$(3)$] $\cW(M)$ is finite.
\item[$(4)$] $\cR(M):=\cR(M)(\cC(M),(\rsys^{[N]})_{N\in \ffg(M)})$ is finite.
\end{itemize}
In all cases, $\cR(M)$ is the unique root system of type $\cC(M)$.
\end{theorem}
\begin{proof}
   Since $M$ is a $\theta$-tuple of one-dimensional \YD modules, we obtain that the Nichols algebra $\cB(M)$ generated by $V_M$ is decomposable and hence $\rsys^{[M]}$ is defined. Then $\cR(M)=\cR(M)(\cC(M),(\rsys^{[N]})_{N\in \ffg(M)})$ is the root system of type $\cC(M)$ by \cite[Theorem 6.11]{HS10}. Hence the claim is true by Lemma~\ref{lem:finite}.
 \end{proof}

\section{Finite Cartan graphs of rank 4}\label{se:Rank4Cartan}

 We give the properties of finite connected indecomposable Cartan graphs of rank 4 in Theorem~\ref{Theo:goodnei}, which will be essential for our classification in the next section.

Let $I=\{1,2,3,4\}$ and $\cC = \cC (I, \cX, r, (A^X)_{X \in \cX})$ be a semi-Cartan graph. Recall that a semi-Cartan graph $\cC$ is standard, if $A^X = A^Y$ for all $X, Y \in Ob(\cW(\cC))$.

To classify all finite Cartan graphs, the following fact is necessary~\cite[corollary 5.4]{HS10}. It shows an important property of finite standard Cartan graphs.
\begin{theorem}~\label{thm:CGfinite}
Let $\cC$ be a standard Cartan graph with generalized
Cartan matrix $A = A^N$ for all $N\in Ob(\cW(\cC))$. Let $\cR$ be a root system
of type $\cC$. Then the Cartan graph $\cW(\cR)$ is finite if and only if $A$ is a Cartan matrix of finite type.
\end{theorem}

For the general Cartan graphs, the points of $\cC$ could have many different neighborhoods. In this case, we define the following "good neighborhoods" in order to cover all the finite connected indecomposable Cartan graphs in such way that at least one point of $\cC$ has one of the good neighborhoods.

To avoid confusion, let
$A_4=\begin{pmatrix}2&-1&0&0\\  -1&2&-1&0\\  0&-1&2&-1\\  0&0&-1&2  \end{pmatrix}$,
$B_4=\begin{pmatrix}2&-1&0&0\\  -1&2&-1&0\\  0&-1&2&-1\\  0&0&-2&2  \end{pmatrix}$,\\
$C_4=B_4'=\begin{pmatrix}2&-1&0&0\\  -1&2&-1&0\\  0&-1&2&-2\\  0&0&-1&2  \end{pmatrix}$,
$D_4=\begin{pmatrix}2&-1&0&0\\  -1&2&-1&-1\\  0&-1&2&0\\  0&-1&0&2  \end{pmatrix}$,
and
$F_4=\begin{pmatrix}2&-1&0&0\\  -1&2&-2&0\\  0&-1&2&-1\\  0&0&-1&2  \end{pmatrix}$.

\begin{defn}\label{defA4}
  We say that $X$ has a \textbf{good $A_4$ neighborhood} if there exists a permutation of $I$ and an integer sequence $(a,b)\in \ndN^2$ such that\\
  $A^X=A^{r_1(X)}=A_4$,
  $A^{r_2(X)}=\begin{pmatrix}2&-1&0&0\\-1&2&-1&0\\0&-1&2&-a\\0&0&-1&2\end{pmatrix}$,
  $A^{r_3(X)}=\begin{pmatrix}2&-1&0&0\\-1&2&-1&-1\\0&-1&2&-1\\0&-1&-1&2\end{pmatrix}$, and
  $A^{r_4(X)}=\begin{pmatrix}2&-1&0&0\\-1&2&-1&0\\0&-b&2&-1\\0&0&-1&2\end{pmatrix}$,
  where $(a,b)$ satisfies one of the following.
  \begin{enumerate}
    \item[$(1)$] $(a,b)\in \{(2,1),(2,2)\}$.
    \item[$(2)$] $(a,b)=(1,2)$, $a_{24}^{r_1r_3(X)}=-1$.
    \item[$(3)$] $(a,b)=(1,1)$, $a_{14}^{r_2r_3(X)}=a_{41}^{r_2r_3(X)}\in \{0,-1\}$.
   \end{enumerate}
\end{defn}
\begin{defn}\label{defB4}
  We say that $X$ has a \textbf{good $B_4$ neighborhood} if there is a permutation of $I$ with respect to which\\
  $A^X=A^{r_i(X)}=B_4$, for all $i\in I$ and $a^{r_3r_4(X)}_{24}=-1$.
\end{defn}

We get the following property of the finite connected indecomposable Cartan graphs by computer calculations algorithms.
\begin{theorem}\label{Theo:goodnei}
Let $M:=(\Bbbk x_1,\dots, \Bbbk x_4)$ be a tuple of one-dimensional \YD modules over $\Bbbk G$. Assume that $M$ admits all reflections and $\cC(M) = \cC (I, \cX, r, (A^X)_{X \in \cX})$ be the attached indecomposable semi-Cartan graph to $M$.
If $\rsys^{[M]}$ is finite, then one of the following is true.
\begin{enumerate}
\item[$(1)$] The Cartan graph $\cC(M)$ is standard (of type $A_4$, $B_4$, $C_4$, $D_4$, $F_4$).
\item[$(2)$] Up to equivalence, there exists a point $Y\in \cX$ such that $Y$ has one of the good $A_4$, $B_4$ neighborhoods.
\end{enumerate}
\end{theorem}
\begin{proof}
  If $\rsys^{[M]}$ is finite, then $\cC(M)$ is a finite Cartan graph and it has a unique finite root system by Theorem \ref{thm:rootofR_M}, say $\cR(M)$.
  Assume that $\cR=\cR(\cC,(\rersys{X})_{X\in \cX})$ is the unique root system, where $\rersys{X}$ is the real roots of $X$.
  Moreover, the root system $\cR(M)$ is irreducible by Proposition \ref{prop.indecom}.
  For any $X\in \cX$,
  let $\rersys{X}_{\boldsymbol{+}}$ be the positive roots of $X$.
  By~\cite[Theorem~4.1]{c-Heck14a} there exists a point $X\in \cX$ satisfying that the set $\rersys{X}_{\boldsymbol{+}}$ is in the list of~\cite[Appendix B.2.]{c-Heck14a} up to a permutation of $I$. There are precisely 11 such possible sets of real roots for the rank 4 case.
  We analyze each set of the real roots in the list.
  For point $Y$, we assume that $\rersys{Y}_{\boldsymbol{+}}$ in the list.
  Since the reflection $s_i^Y$ maps $\rersys{Y}_{\boldsymbol{+}}\setminus \{\alpha_i\}$ bijectively to $\rersys{r_i(Y)}_{\boldsymbol{+}}\setminus \{\alpha_i\}$ for any $i\in I$ ,
  the Cartan matrices of all neighbors of $Y$ can be obtained from $\rersys{Y}_{\boldsymbol{+}}$ by Lemma~\ref{lem:jik}.
  If the Cartan graph $\cC(M)$ is standard or $Y$ has a good $A_4$ or $B_4$ neighborhood, then the claim is true.
  Otherwise repeat the previous step to the neighbours of $Y$. Since $\cX$ is finite, this algorithm terminates.
  The elementary calculations are done by GAP algorithms and they are skipped here.
\end{proof}

\section{Classification theorem for rank 4 case}\label{se:clasi}
In this section, all rank 4 Nichols algebras of diagonal type with a finite set of roots are determined. We formulate the main result in Theorem~\ref{theo:clasi} and present the corresponding generalized Dynkin diagrams in Table~\ref{tab.1}.

\begin{theorem}\label{theo:clasi}
Let $\Bbbk$ is a field of characteristic $p>0$. Let $I=\{1,2,3,4\}$.
Let $(V,c)$ be a braided vector space of diagonal type over $\Bbbk$ with basis $\{x_k|k\in I\}$ satisfying
\begin{equation*}
  c(x_i \otimes x_j) = q_{ij}x_j \otimes x_i \quad \textit{for some} \quad q_{ij} \in \Bbbk^*.
\end{equation*}
Assume that the $(q_{ij})_{i,j\in I}$ is an indecomposable braiding matrix.
Let $M:=(\Bbbk x_i)_{i\in I}$.
Then the Nichols algebra $\cB(V)$ generated by $(V,c)$ has a finite set of roots ${\roots}^{[M]}$ if and only if
the generalized Dynkin diagram $\cD$ of $V$ appears in Table~\ref{tab.1}.
In this case,
the row of Table~\ref{tab.1} containing $\cD$
consists precisely of the generalized Dynkin diagrams of all the points of $\cC(M)$.
The corresponding row of Table~\ref{tab.2} contains the exchange graph of $\cC(M)$.
\end{theorem}
We claim that Theorem~\ref{theo:clasi} is enough to classify finite-dimensional Nichols algebra of diagonal type by~\cite[Corollary~6]{a-Heck04e}.
\begin{cor}\label{coro-cla}
  Assume that $p>0$.
  Then the Nichols algebra $\cB(V)$ is finite dimensional if and only if the generalized Dynkin diagram $\cD$ of $V$ appears in Table~(\ref{tab.1}) and the labels of the vertices of $\cD$ are roots of unity (including $1$).
\end{cor}
\begin{rem}
In the part of the proof, we give the following statement to avoid confusion.
 \begin{itemize}
  \item[$(1)$]
Label the vertices of the generalized Dynkin diagrams from left to right and then top to bottom by $1, \ldots, 4$
(for example: label $D_4$ as in Fig.\ \ref{fig_D4}).
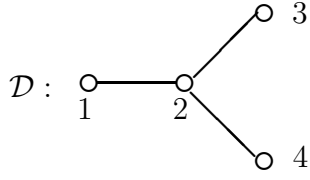
\begin{figure}[h!]
%
 \begin{center}
\setlength{\unitlength}{3947sp}
\begingroup\makeatletter\ifx\SetFigFont\undefined
\gdef\SetFigFont#1#2#3#4#5{%
  \reset@font\fontsize{#1}{#2pt}%
  \fontfamily{#3}\fontseries{#4}\fontshape{#5}%
  \selectfont}%
\fi\endgroup
\begin{picture}(3348,928)(500,-705)
\thicklines
\put(700,-300){$\cD:$}
{\put(1250,-211){\line( 1, 0){500}}}
{\put(1860,-170){\line( 1, 1){400}}}
{\put(1840,-270){\line( 1, -1){400}}}
\put(1726,-436){\makebox(0,0)[lb]{\smash{{\SetFigFont{12}{14.4}{\rmdefault}{\mddefault}{\updefault}{$2$}}}}}
\put(2480,-736){\makebox(0,0)[lb]{\smash{{\SetFigFont{12}{14.4}{\rmdefault}{\mddefault}{\updefault}{$4$}}}}}
\put(1126,-436){\makebox(0,0)[lb]{\smash{{\SetFigFont{12}{14.4}{\rmdefault}{\mddefault}{\updefault}{$1$}}}}}
\put(2476,164){\makebox(0,0)[lb]{\smash{{\SetFigFont{12}{14.4}{\rmdefault}{\mddefault}{\updefault}{$3$}}}}}
{\put(1201,-211){\circle{100}}}
{\put(1801,-211){\circle{100}}}
{\put(2300,230){\circle{100}}} 
{\put(2300,-701){\circle{100}}}
\end{picture}
 \end{center}
\caption{generalized Dynkin diagram of type $D_4$}
\label{fig_D4}
\end{figure}
\item[$(2)$] For a generalized Dynkin diagram $\cD$ and $i,j,k,l\in I$,
write $\tau_{ijkl} \cD$ for the graph $\cD$
where the vertices of $\cD$ change to $i$, $j$, $k$, $l$ respectively.
\end{itemize}
\end{rem}

\begin{proof}
We prove the theorem by the following two steps:
  \begin{enumerate}
 \item[(1)]
  The if part is clear. Indeed, assume that the generalized Dynkin diagram $\cD$ appears in row $r$ of any of Table~\ref{tab.1}.
  From Lemmas~\ref{lem:aij} and~\ref{le:Dynkin} we determine that $M$ admits all reflections.
  The detailed calculations are skipped at this point here.
  Hence the Cartan graph $\cC(M)$ can be defined by theorem~\ref{theo.regualrcar}.
  Notice that $\cC(M)$ is the same with the Cartan graph obtained from the Dykin diagrams listed in row $s$ of Table~3 in~\cite{a-Heck09},
  where $s$ appears in the third column of row $r$ of Table~\ref{tab.2}.
  Moreover,
  the arithmetic root systems of the above Cartan graphs are finite,
  see~\cite[Theorem~17]{a-Heck09}.
  Hence $\cW(\cC(M))$ is finite.
  Then $\cB(V)$ has a finite set of roots ${\roots}^{[M]}$ by Theorem~\ref{thm:rootofR_M}.

  \item[(2)]
   Next we prove that if $\cB(V)$ has a finite set of roots then the generalized Dynkin diagram of $V$ appears in Table~\ref{tab.1}.
   we assume that $\cB(V)$ has a finite set of roots ${\roots}^{[M]}$. Let $X=[M]^s_4$, $I=\{1,2,3,4\}$, and $A^X:=(a_{ij})_{i,j\in I}$ be the Cartan matrix of $X$.
  Let $(q_{i,j})_{i,j\in I}$ be the braiding matrix of $X$ and $(q_{i,j}^{r_i(X)})_{i,j\in I}$ be the braiding matrix of $r_i(X)$.
  To simply the labels,
  we write $q_{ij}':=q_{ij}q_{ji}$ for $1\leq i,j\leq 4$.
  Since $\cB(V)$ has a finite set of roots ${\roots}^{[M]}$, we obtain that $\cC(M)$ is a finite Cartan graph by Theorem~\ref{thm:rootofR_M}.
  we are free to assume that either $\cC(M)$ is standard or there exists a point $X$ such that $A^X$ has a good $A_4$ or $B_4$ neighborhood by Theorem~\ref{Theo:goodnei}.

  Case a. Assume that either $X$ has a good $A_4$ neighborhood or $\cC(M)$ is standard of type $A_4$.
  Let $(a,b):=(-a_{34}^{r_2(X)}, -a_{32}^{r_4(X)})$.
  From~Lemma~\ref{le:Dynkin} the condition $A^X=A_4$ implies that $(2)_{q_{ii}}(q_{ii}q_{i,i+1}'-1)=(2)_{q_{jj}}(q_{jj}q_{j-1,j}'-1)=0$,
  for all $i\in \{1,2,3\}$ and $j\in \{2,3,4\}$.
  Hence we distinguish the following subcases: $a_1$, $a_2$, $\dots$, $a_{18}$.

  Subcase $a_1$.
  Consider that $q_{ii}q_{i,i+1}'-1=q_{jj}q_{j-1,j}'-1=0$ for all $i\in \{1,2,3\}$ and $j\in \{2,3,4\}$.
  Then $\cC(M)$ is standard and $\cD=\cD_{11}$.

  Subcase $a_2$.
  Consider that 
  $q_{11}=-1$ and $q_{ii}q_{i,i+1}'-1=q_{jj}q_{j-1,j}'-1=0$ for all $i\in \{2,3\}$ and $j\in \{2,3,4\}$.
  To avoid the repetition we assume that $q_{12}'\not=-1$.
  Otherwise we get the case in $a_1$.
  Then $\cC(M)$ is standard and $\cD=\cD_{61}$.

  Subcase $a_3$.
  Consider that
  $q_{22}=-1$ and $q_{ii}q_{i,i+1}'-1=q_{jj}q_{j-1,j}'-1=0$ for all $i\in \{1,3\}$ and $j\in \{3,4\}$.
  Assume that the condition $\{q_{12}',q_{23}'\}=\{-1\}$ doesn't hold in this case to avoid the repetition.
  Then by~\cite[Lemma 1.4]{W-17} for $a_{13}^{r_2(X)}=0$
  we get that $q_{12}'q_{23}'=1$.
  Then $q_{12}'\not=-1$ and $\cC(M)$ is standard.
  Hence $\cD=\cD_{10,1}$.

  Subcase $a_4$.
  Consider that
  $q_{33}=-1$ and $q_{ii}q_{i,i+1}'-1=q_{jj}q_{j-1,j}'-1=0$ for all $i\in \{1,2\}$ and $j\in \{2,4\}$.
  Assume that the condition $\{q_{23}',q_{34}'\}=\{-1\}$ doesn't hold in this case to avoid repetition.
  From~\cite[Lemma 1.4]{W-17} we get $(a,b)=(1,1)$.
  Let $q:=q_{23}'$ and $r:=q_{34}'$.
  Then the condition $\{q,r\}=\{-1\}$ doesn't hold.
  If $\cC(M)$ is standard, then $qr=1$ and $q\not=-1$ by~\cite[Lemma 1.4]{W-17}.
  Hence $\cD=\tau_{4321}\cD_{10,1}$.
  If $X$ has a good $A_4$ neighborhood,
  then $a_{14}^{r_2r_3(X)}=a_{41}^{r_2r_3(X)}\in \{0,-1\}$ by Definition~\ref{defA4}(3).
  If $a_{14}^{r_2r_3(X)}=0$, then $q^2r=1$, $q\not=-1$, and hence $\cD=\cD_{13,1}$.
  If $a_{14}^{r_2r_3(X)}=a_{41}^{r_2r_3(X)}=-1$, then
  $q^2r\not=1$ and the reflections of $X$
\setlength{\unitlength}{1mm}
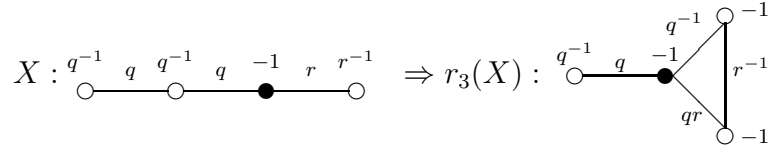
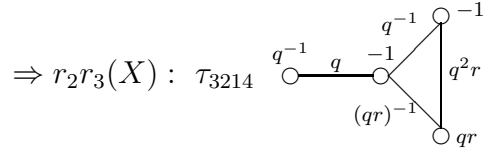
\begin{align*}
X:~
\Dchainfour{3}{$q^{-1}$}{$q$}{$q^{-1}$}{$q$}%
{$-1$}{$r$}{$r^{-1}$}
\quad \Rightarrow
r_3(X):~
\Drightofway{m}{$q^{-1}$}{$q$}{$-1$}{$q^{-1}$}{$qr$}%
{$-1$}{$r^{-1}$}{$-1$}
\end{align*}
\begin{align*}
\quad \Rightarrow
r_2r_3(X):~
\tau_{3214} \ \ \Drightofway{}{$q^{-1}$}{$q$}{$-1$}{$q^{-1}$}{$(qr)^{-1}$}%
{$-1$}{$q^2r$}{$qr$}
\end{align*}
imply that $qr=-1$ or $q^3r^2=1$ by~Lemma~\ref{le:Dynkin}.
If $qr=-1$, then $q\not=-1$, $p\not=2$, and $\cD=\cD_{14,1}$.
If $q^3r^2=1$, then $\cD=\cD_{94}$.

  Subcase $a_5$.
  Consider that
  $q_{44}=-1$, $q_{34}'\not=-1$, and $q_{ii}q_{i,i+1}'-1=q_{jj}q_{j-1,j}'-1=0$ for all $i\in \{1,2,3\}$ and $j\in \{2,3\}$.
  Then $\cC(M)$ is standard and $\cD=\tau_{4321}\cD_{61}$.

  Subcase $a_6$.
  Consider that
  $q_{11}=q_{22}=-1$, $q_{12}'\not=-1$, and $q_{33}q_{34}'-1=q_{ii}q_{i-1,i}'-1=0$ for all $i\in \{3,4\}$.
  Then $A^{r_3(X)}=A_4$ and hence $\cC(M)$ is standard by the assumption.
  Then $q_{12}'q_{23}'=1$ by~\cite[Lemma 1.4]{W-17} and hence $\cD=\cD_{62}$.

  Subcase $a_7$.
  Consider that
  $q_{11}=q_{33}=-1$, $q_{12}'\not=-1$, and $q_{22}q_{23}'-1=q_{ii}q_{i-1,j}'-1=0$ for all $i\in \{2,4\}$.
  Then we get $(a,b)=(1,1)$.
  Let $q:=q_{23}'$ and $r:=q_{34}'$.
  Then $q\not=-1$.
  If $\cC(M)$ is standard,
  then $qr=1$ and $\cD=\cD_{10,3}$.
  If $X$ has a good $A_4$ neighborhood,
  then $a_{41}^{r_2r_3(X)}=a_{14}^{r_2r_3(X)}\in \{0,1\}$ from Definition~\ref{defA4}(3).
  If $a_{41}^{r_2r_3(X)}=0$,
  then $q^2r=1$, $q\not=-1$, and hence $\cD=\cD_{12,3}$.
  If $a_{14}^{r_2r_3(X)}=a_{41}^{r_2r_3(X)}=-1$, then $q^2r\not=1$
  and the reflections of $X$
\setlength{\unitlength}{1mm}
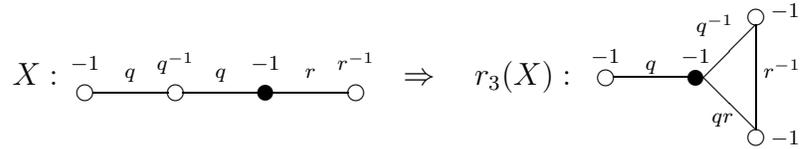
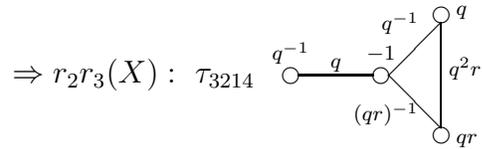
\begin{align*}
X:~
\Dchainfour{3}{$-1$}{$q$}{$q^{-1}$}{$q$}%
{$-1$}{$r$}{$r^{-1}$}
\quad \Rightarrow \quad
r_3(X):~
\Drightofway{m}{$-1$}{$q$}{$-1$}{$q^{-1}$}{$qr$}%
{$-1$}{$r^{-1}$}{$-1$}
\end{align*}
\begin{align*}
\quad \Rightarrow
r_2r_3(X):~
\tau_{3214} \ \ \Drightofway{}{$q^{-1}$}{$q$}{$-1$}{$q^{-1}$}{$(qr)^{-1}$}%
{$q$}{$q^2r$}{$qr$}
\end{align*}
show that $q^3r=1$, $qr=-1$, and $p\not=2$.
Hence $\cD=\cD_{22,3}$.

   Subcase $a_8$.
  Consider that
  $q_{11}=q_{44}=-1$ and $q_{ii}q_{i,i+1}'-1=q_{ii}q_{i-1,i}'-1=0$ for all $i\in \{2,3\}$.
  To avoid repetition we assume that $q_{12}', q_{34}'\not=-1$.
  Then $\cC(M)$ is standard and $\cD=\cD_{10,5}$.

  Subcase $a_9$.
  Consider that
  $q_{22}=q_{33}=-1$ and $q_{11}q_{12}'-1=q_{44}q_{34}'-1=0$.
  Set $q:=q_{11}$, $r:=q_{23}'$, and $s:=q_{44}$.
  To avoid repetition we assume that the condition $\{q,r\}=\{-1\}$ doesn't hold.
  We obtain that $b=1$ by~Lemma~\ref{le:Dynkin}.
  If $\cC(M)$ is standard,
  then $q=r=s\not=-1$, and $\cD=\cD_{63}$.
  If $X$ has a good $A_4$ neighborhood,
  then $r=q\not=s$ and $(a,b)\in \{(1,1),(2,1)\}$ by Definition~\ref{defA4}.
  Hence the reflections of $X$
\setlength{\unitlength}{1mm}
\begin{align*}
r_3(X):~
\Drightofway{t}{$r$}{$r^{-1}$}{$r$}{$r^{-1}$}{$rs^{-1}$}%
{$-1$}{$s$}{$-1$} \Rightarrow
X:~
\Dchainfour{2}{$r$}{$r^{-1}$}{$-1$}{$r$}%
{$-1$}{$s^{-1}$}{$s$}
\end{align*}
\begin{align*}
 \quad \Rightarrow
r_2(X):~
\Dchainfour{}{$-1$}{$r$}{$-1$}{$r^{-1}$}%
{$r$}{$s^{-1}$}{$s$}
\end{align*}
imply that $a=2$ by $r\notin \{-1,s\}$ and $s=r^2$ by $a_{24}^{r_3(X)}=-1$.
Then $\cD=\cD_{83}$.

Subcase $a_{10}$.
Consider that 
  $q_{22}=q_{44}=-1$, $q_{34}'\not=-1$, and $q_{ii}q_{i,i+1}'-1=q_{33}q_{23}'-1=0$ for all $i\in \{1,3\}$.
  To avoid repetition we assume that the condition $\{q_{12}',q_{23}'\}=\{-1\}$ doesn't hold.
Then $\cC(M)$ is standard by the assumption and $A^{r_3(X)}=A_4$.
Hence $q_{12}'q_{23}'=1$ by~\cite[Lemma 1.4]{W-17} and $\cD=\tau_{4321}\cD_{10,3}$.

Subcase $a_{11}$.
Consider that
 $q_{33}=q_{44}=-1$, $q_{34}'\not=-1$, and $q_{ii}q_{i,i+1}'-1=q_{22}q_{12}'-1=0$ for all $i\in \{1,2\}$.
 Let $q:=q_{11}$ and $r:=q_{34}'$.
 Then $r\not=-1$.
 If $\cC(M)$ is standard,
 then $q=r$ by~\cite[Lemma 1.4]{W-17} and hence $\cD=\tau_{4321}\cD_{62}$.
 If $X$ has a good $A_4$ neighborhood,
 then $a=1$, $r\not=q$, and hence $(a,b)\in \{(1,1),(1,2)\}$ by Definition~\ref{defA4}.
 The reflections of $X$
 \setlength{\unitlength}{1mm}
\begin{align*}
r_4(X):~
\Dchainfour{4}{$q$}{$q^{-1}$}{$q$}{$q^{-1}$}%
{$r$}{$r^{-1}$}{$-1$}
\quad \Rightarrow ~
X:~
 \Dchainfour{3}{$q$}{$q^{-1}$}{$q$}{$q^{-1}$}%
{$-1$}{$r$}{$-1$}
\end{align*}
\begin{align*}
\quad \Rightarrow ~
r_3(X):~
\Drightofway{}{$q$}{$q^{-1}$}{$-1$}{$q$}{$rq^{-1}$}%
{$-1$}{$r^{-1}$}{$r$}
\end{align*}
show that $b=2$ by the condition $r\notin \{-1, q\}$.
Then we obtain that $r^2q^{-1}=1$ by $a_{42}^{r_3(X)}=-1$ and $(3)_r(r^2q^{-1}-1)=0$ by $b=2$.
If $r^2q^{-1}=1$ and $r^3\not=1$,
then $\cD=\cD_{92}$.
If $r^2q^{-1}=1$ and $r^3=1$,
then $p\not=3$, $r\in G_3'$, $qr=1$, and $\cD=\cD_{18,2}$.

Subcase $a_{12}$.
Consider that
 $q_{11}=q_{22}=q_{33}=-1$, $q_{12}'\not=-1$, and $q_{44}q_{34}'-1=0$.
 Let $q:=q_{12}'$, $r:=q_{23}'$, and $s:=q_{34}'$.
 Assume that the equations $\{r,s\}=\{-1\}$ doesn't hold.
 Then $A^{r_4(X)}=A_4$ and $b=1$.
 If $\cC$ is standard,
 then $qr=rs=1$ and hence $\cD=\cD_{10,2}$.
 If $X$ has a good $A_4$ neighborhood,
 then $qr=1$, $rs\not=1$ and the reflections of $X$ are
 \setlength{\unitlength}{1mm}
\begin{align*}
r_3(X):~
\Drightofway{t}{$-1$}{$q$}{$q^{-1}$}{$q$}{$sq^{-1}$}%
{$-1$}{$s^{-1}$}{$-1$}
\quad \Rightarrow
X:~
\Dchainfour{2}{$-1$}{$q$}{$-1$}{$q^{-1}$}%
{$-1$}{$s$}{$s^{-1}$}
\end{align*}
\begin{align*}
\quad \Rightarrow
r_2(X):~
\Dchainfour{}{$q$}{$q^{-1}$}{$-1$}{$q$}%
{$q^{-1}$}{$s$}{$s^{-1}$}
\end{align*}
with $q\notin \{-1,s\}$.
Then we obtain $a=2$ by~Lemma~\ref{le:Dynkin} and $s=q^2$ by $a_{24}^{r_{3}(X)}=-1$.
Hence $\cD=\cD_{12,2}$.

Subcase $a_{13}$.
Consider that
 $q_{11}=q_{22}=q_{44}=-1$, $q_{12}'\not=-1$, $q_{34}'\not=-1$, and $q_{33}q_{23}'-1=q_{33}q_{34}'-1=0$.
 Then $\cC$ is standard by $A^{r_3(X)}=A_4$ and the assumption.
 Hence $q_{12}'q_{23}'=1$ and $\cD=\tau_{4321}\cD_{10,4}$.

 Subcase $a_{14}$.
Consider that $q_{11}=q_{33}=q_{44}=-1$, $q_{12}'\not=-1$, $q_{34}'\not=-1$, and $q_{22}q_{12}'-1=q_{22}q_{23}'-1=0$.
Let $q:=q_{12}'$ and $r:=q_{34}'$.
Then $q,r\not=-1$.
Then $a=1$ by~Lemma~\ref{le:Dynkin}.
If $\cC$ is standard,
then $qr=1$ and $\cD=\cD_{10,4}$.
If $X$ has a good $A_4$ neighborhood,
then $qr\not=1$ and the reflections of $X$ are
 \setlength{\unitlength}{1mm}
 \begin{align*}
r_4(X):~
\Dchainfour{4}{$-1$}{$q$}{$q^{-1}$}{$q$}%
{$r$}{$r^{-1}$}{$-1$}
\quad \Rightarrow
X:~
\Dchainfour{3}{$-1$}{$q$}{$q^{-1}$}{$q$}%
{$-1$}{$r$}{$-1$}
\end{align*}
\setlength{\unitlength}{1mm}
 \begin{align}\label{diaa14}
\quad \Rightarrow
r_3(X):~
\Drightofway{l}{$-1$}{$q$}{$-1$}{$q^{-1}$}{$qr$}%
{$-1$}{$r^{-1}$}{$r$}
\quad \Rightarrow
r_1r_3(X):~
\Drightofway{}{$-1$}{$q^{-1}$}{$q$}{$q^{-1}$}{$qr$}%
{$-1$}{$r^{-1}$}{$r$}
\end{align}
with $q, r\not=-1$.
Then $qr^2=1$ by $a_{42}^{r_3(X)}=-1$ and hence $b=2$ by~Lemma~\ref{le:Dynkin}.
Hence $(a,b)=(1,2)$ and $a_{24}^{r_1r_3(X)}=-1$ by Definition~\ref{defA4}(2).
Then $q^2r=1$ by the reflections~(\ref{diaa14}) and~Lemma~\ref{le:Dynkin}.
Hence $r=q\in G_3'$, $p\not=3$, and $\cD=\cD_{20,4}$.

Subcase $a_{15}$.
Consider that $q_{22}=q_{33}=q_{44}=-1$, $q_{34}'\not=-1$, and $q_{11}q_{12}'-1=0$.
Let $q:=q_{12}'$, $r:=q_{23}'$, and $s:=q_{34}'$.
Assume that the condition $\{q, r\}=\{-1\}$ doesn't hold.
If $\cC$ is standard,
then $qr=rs=1$ and $\cD=\tau_{4321}\cD_{10,2}$.
If $X$ has a good $A_4$ neighborhood,
then
$qr=1$, $rs\not=1$, $r\not=-1$, and the reflections of $X$ are
 \setlength{\unitlength}{1mm}
 \begin{align*}
X:~
\Dchainfour{3}{$r$}{$r^{-1}$}{$-1$}{$r$}%
{$-1$}{$s$}{$-1$}
\quad \Rightarrow \quad
r_3(X):~
\Drightofway{}{$r$}{$r^{-1}$}{$r$}{$r^{-1}$}{$rs$}%
{$-1$}{$s^{-1}$}{$s$}
\end{align*}
with $r,s\not=-1$.
We obtain the equation $rs^2=r^2s=1$ by~Lemma~\ref{le:Dynkin} on $a_{24}^{r_3(X)}=a_{42}^{r_3(X)}=-1$.
Then $s=r\in G_3'$, $p\not=3$, and $a=b=2$.
Hence $\cD=\cD_{21,6}$.

Subcase $a_{16}$.
Consider that $q_{11}=q_{22}=q_{33}=q_{44}=-1$, $q_{12}'\not=-1$, and $q_{34}'\not=-1$.
Let $q:=q_{12}'$, $r:=q_{23}'$, and $s:=q_{34}'$.
We assume that $q,s\not=-1$ to avoid repetition.
If $\cC$ is standard,
then $qr=rs=1$, $q\not=-1$, and hence $\cD=\cD_{10,6}$.
If $X$ has a good $A_4$ neighborhood,
then $qr=1$, $sr\not=1$,  and the reflection of $X$ is
 \setlength{\unitlength}{1mm}
 \begin{align*}
X:~
\Dchainfour{3}{$-1$}{$r^{-1}$}{$-1$}{$r$}%
{$-1$}{$s$}{$-1$}
\quad \Rightarrow \quad
r_3(X):~
\Drightofway{}{$-1$}{$r^{-1}$}{$r$}{$r^{-1}$}{$rs$}%
{$-1$}{$s^{-1}$}{$s$}
\end{align*}
with $r, s\not=-1$.
Then we get $rs^2=r^2s=1$ by~Lemma~\ref{le:Dynkin} on $a_{24}^{r_3(X)}=a_{42}^{r_3(X)}=-1$.
Then $s=r\in G_3'$, $p\not=3$, and hence $\cD=\cD_{20,3}$.

  Case (b). Assume that either $X$ has a good $B_4$ neighborhood or $\cC(M)$ is standard of type $B_4$.
  Then by~Lemma~\ref{le:Dynkin} on $A^X=B_4$ we get the equations
  $$(2)_{q_{ii}}(q_{ii}q_{i,i+1}'-1)=(2)_{q_{jj}}(q_{jj}q_{j-1,j}'-1)=(3)_{q_{44}}(q_{44}^2q_{34}'-1)=0,$$
  for all $i\in \{1,2,3\}$ and $j\in \{2,3\}$.
  Then we distinguish the subcases: $a_1$, $a_2$, $\dots$, $a_{14}$ by~Lemma~\ref{le:Dynkin}.

  Subcase $b_1$.
  Consider that $q_{ii}q_{i,i+1}'-1=q_{jj}q_{j-1,j}'-1=q_{44}^2q_{34}'-1=0$ for all $i\in \{1,2,3\}$ and $j\in \{2,3\}$.
  Then $\cC(M)$ is standard and $\cD=\cD_{21}$.

  Subcase $b_2$.
  Consider that $q_{11}=-1$ and $q_{ii}q_{i,i+1}'-1=q_{ii}q_{i-1,i}'-1=q_{44}^2q_{34}'-1=0$ for all $i\in \{2,3\}$.
  In this case we assume that $q_{12}'\not=-1$ to avoid the repetition.
  Then $\cC(M)$ is standard and $\cD=\cD_{71}$.

  Subcase $b_3$.
  Consider that there exists $k\in \{2,3\}$ such that $q_{kk}=-1$ and $q_{ii}q_{i,i+1}'-1=q_{jj}q_{j-1,j}'-1=q_{44}^2q_{34}'-1=0$ for any $i\in \{1,2,3\}\setminus \{k\}$ and $j\in \{2,3\}\setminus \{k\}$.
  We assume that the condition $\{q_{k,k-1}', q_{k,k+1}'\}=\{-1\}$ doesn't hold in this case to avoid repetition.
  We get $a_{13}^{r_k(X)}=0$ by assumption.
  Then $q_{k,k-1}'q_{k,k+1}'=1$ by~\cite[Lemma 1.4]{W-17} and hence $q_{k,k-1}'\not=-1$.
  Then $\cC$ is standard.
  We obtain that if $k=2$ then $\cD=\cD_{11,1}$ and if $k=3$ then  $\cD=\cD_{74}$.

  Subcase $b_4$
  Consider that $(3)_{q_{44}}=0$
  and $q_{ii}q_{i,i+1}'-1=q_{jj}q_{j-1,j}'-1=0$ for all $i\in \{1,2,3\}$ and $j\in \{2,3\}$.
  We assume that $q_{44}^2q_{34}'-1\not=0$ to avoid repetition.
  Set $q:=q_{22}$ and $\zeta:=q_{33}$.
  The reflection of $X$ is
  \setlength{\unitlength}{1mm}
 \begin{align*}
X:
 \Dchainfour{4}{$q$}{$q^{-1}$}{$q$}{$q^{-1}$}%
{$q$}{$q^{-1}$}{$\zeta$}
\quad \Rightarrow \quad
r_4(X):
\Dchainfour{}{$q$}{$q^{-1}$}{$q$}{$q^{-1}$}%
{$\zeta q^{-1}$}{$q\zeta^{-1}$}{$\zeta$}
\end{align*}
with $(3)_{\zeta}=0$ and $q\notin \{\zeta, \zeta^{-1}\}$.
we get that $\zeta q^{-2}=1$ or $\zeta q^{-1}=-1$ by~Lemma~\ref{le:Dynkin} on $A^{r_4(X)}=B_4$.
Hence $q=-\zeta^{-1}$, $p\not=2$ or $q=-\zeta$, $p\not=2$.
If $q=-\zeta^{-1}$ and $p\not=2, 3$,
then $\cC(M)$ is standard and $\cD=\cD_{15,1}$.
If $q=-\zeta$ and $p\not=2,3$,
then $X$ has a good $B_4$ neighborhood and $\cD=\cD_{17,1}$.
If $p=3$,
then we get $q=-1$ for $q=-\zeta^{-1}$ and $q=-\zeta$.
Then $\cC(M)$ is standard and $\cD=\cD_{15',1}$.

Subcase $b_5$.
  Consider that there exists $k\in \{2,3\}$ such that $q_{kk}=q_{11}=-1$ and $q_{44}^2q_{34}'-1=q_{jj}q_{j,j+1}'-1=q_{jj}q_{j-1,j}'-1=0$ for $j\in \{2,3\}\setminus \{k\}$.
  We assume that $q_{12}'\not=-1$.
  Then $q_{k,k-1}'q_{k,k+1}'=1$ by $A^{r_k(X)}=B_4$.
  Hence $\cC$ is standard.
  If $k=2$ then $\cD=\cD_{72}$.
  If $k=3$ then $\cD=\cD_{11,3}$.

  Subcase $b_6$.
  Consider that $q_{11}=-1$, $(3)_{q_{44}}=0$, $q_{12}'\not=-1$, $q_{44}^2q_{34}'-1\not=0$, and $q_{ii}q_{i,i+1}'-1=q_{ii}q_{i-1,i}'-1=0$, for all $i\in \{2,3\}$.
  Let $q:=q_{22}$ and $\zeta:=q_{44}$.
  Then $(3)_{\zeta}=0$ and $q\notin \{-1,\zeta,\zeta^{-1}\}$.
  By~Lemma~\ref{le:Dynkin} on $A^{r_4(X)}=B_4$ the reflections of $X$
  \setlength{\unitlength}{1mm}
 \begin{align*}
 r_1(X):
 \Dchainfour{1}{$-1$}{$q$}{$-1$}{$q^{-1}$}%
{$q$}{$q^{-1}$}{$\zeta$}
 \Rightarrow
 X:
 \Dchainfour{4}{$-1$}{$q^{-1}$}{$q$}{$q^{-1}$}%
{$q$}{$q^{-1}$}{$\zeta$}
\end{align*}
\begin{align*}
\Rightarrow
r_4(X):
\Dchainfour{}{$-1$}{$q^{-1}$}{$q$}{$q^{-1}$}%
{$\zeta q^{-1}$}{$q\zeta^{-1}$}{$\zeta$}
\end{align*}
imply that $\zeta q^{-2}=1$ or $\zeta q^{-1}=-1$.
Then $q=-\zeta^{-1}$, $p\not=2$ or $q=-\zeta$, $p\not=2$.
If $q=-\zeta^{-1}$ and $p\not=2, 3$,
then $\cC(M)$ is standard and $\cD=\cD_{16,1}$.
If $q=-\zeta^{-1}$ and $p=3$,
then $q=-1$ and $\cD=\cD_{15',1}$.
If $q=-\zeta$,
then one of the the generalized Dynkin subdiagram of $r_1(X)$ is~
\Dchainthree{}{$-1$}{$-\zeta^{-1}$}%
{$-\zeta$}{$-\zeta^{-1}$}{$\zeta$}~,
which does not appear in~\cite[Tables 1-3]{W-17}.
Hence we get a contradiction.

Subcase $b_7$.
  Consider that $q_{22}=q_{33}=-1$ and $q_{44}^2q_{34}'-1=q_{11}q_{12}'-1=0$.
  Set $q:=q_{12}'$, $r:=q_{23}'$, and $s:=q_{34}'$.
  We assume that the condition $\{q, r\}=\{-1\}$ doesn't hold and neither does the condition $\{r,s\}=\{-1\}$.
  Then by~\cite[Lemma 1.4]{W-17} for $A^{r_{2}(X)}=A^{r_{3}(X)}=B_4$ we get $qr=rs=1$.
  Hence $\cC$ is standard and $\cD=\cD_{73}$.

Subcase $b_8$.
  Consider that $q_{22}=-1$, $(3)_{q_{44}}=0$, $q_{44}^2q_{34}'-1\not=0$, and $q_{ii}q_{i,i+1}'-1=q_{33}q_{23}'-1=0$ for all $i\in \{1,3\}$.
  Set $q:=q_{12}'$, $r:=q_{23}'$, and $\zeta:=q_{44}$.
  Assume that $\{q,r\}=\{-1\}$ doesn't hold.
  Then $(3)_{\zeta}=0$ and $r\notin \{\zeta,\zeta^{-1}\}$.
  Then by~\cite[Lemma 1.4]{W-17} on $A^{r_2(X)}=B_4$ we get $qr=1$ and $r\not=-1$.
  Then by~Lemma~\ref{le:Dynkin} the reflections of $X$
  \setlength{\unitlength}{1mm}
 \begin{align*}
 X:~
 \Dchainfour{4}{$q^{-1}$}{$q$}{$-1$}{$q^{-1}$}%
{$q$}{$q^{-1}$}{$\zeta$}
\quad \Rightarrow \quad
r_4(X):~
\Dchainfour{}{$q^{-1}$}{$q$}{$-1$}{$q^{-1}$}%
{$\zeta q^{-1}$}{$q\zeta^{-1}$}{$\zeta$}
\end{align*}
imply that $\zeta q^{-2}=1$ or $\zeta q^{-1}=-1$ by $A^{r_4(X)}=B_4$.
Hence $q=-\zeta^{-1}$, $p\not=2$ or $q=-\zeta$, $p\not=2$.
If $q=-\zeta^{-1}$ and $p\not=2,3$,
then $\cC(M)$ is standard and $\cD=\cD_{19,1}$.
If $q=-\zeta^{-1}$ and $p=3$,
then $q=-1$ and $\cD=\cD_{15',1}$.
If $q=-\zeta$,
then one of the the generalized Dynkin subdiagram of $X$ is
\Dchainthree{}{$-1$}{$-\zeta^{-1}$}%
{$-\zeta$}{$-\zeta^{-1}$}{$\zeta$}~,
which does not appear in~\cite[Tables 1-3]{W-17}.
Hence we get a contradiction.

 Subcase $b_{9}$.
 Consider that $q_{33}=-1$, $(3)_{q_{44}}=0$, $q_{44}^2 q_{34}'-1\not=0$, and $q_{ii}q_{i,i+1}'-1=q_{22}q_{12}'-1=0$, for all $i\in \{1,2\}$.
 In this case we assume that the condition $\{q_{23}, q_{34}'\}=\{-1\}$ doesn't hold.
 Set $q:=q_{11}$, $r:=q_{34}'$, and $\zeta:=q_{44}$.
 We obtain that $q=r$ by $A^{r_3(X)}=B_4$.
 Then $(3)_{\zeta}=0$ and $q\notin \{-1, \zeta, \zeta^{-1}\}$.
 Hence the reflection of $X$
 \setlength{\unitlength}{1mm}
 \begin{align*}
X:~
\Dchainfour{4}{$q$}{$q^{-1}$}{$q$}{$q^{-1}$}%
{$-1$}{$q$}{$\zeta$}
\quad \Rightarrow \quad
r_4(X):~
\Dchainfour{}{$q$}{$q^{-1}$}{$q$}{$q^{-1}$}%
{$-\zeta q^2$}{$(\zeta q)^{-1}$}{$\zeta$}
\end{align*}
imply that $\zeta q^2=1$ or $-\zeta q^2 (\zeta q)^{-1}= -\zeta q^2 q^{-1}=1$ by~Lemma~\ref{le:Dynkin} on $A^{r_4(X)}=B_4$.
Then $q=-\zeta$, $p\not=2,3$.
 Hence $\cC$ is standard and $\cD=\cD_{16,4}$.

Subcase $b_{10}$.
 Consider that $q_{11}=q_{22}=q_{33}=-1$, $q_{44}^2 q_{34}'-1=0$, and $q_{12}'\not=-1$.
 Assume that the condition $\{q_{23}',q_{34}'\}=\{-1\}$ doesn't hold in this case.
 Set $q:=q_{12}'$, $r:=q_{23}'$, and $s:=q_{44}$.
 Then we get that $qr=1$, $r=s^2\notin \{1,-1\}$ by $A^{r_2(X)}=A^{r_3(X)}=B_4$.
 Hence $\cD=\cD_{11,2}$.

  Subcase $b_{11}$.
 Consider that $q_{11}=q_{22}=-1$, $(3)_{q_{44}}=0$, $q_{12}'\not=-1$, $q_{44}^2 q_{34}'-1\not=0$, and $q_{33}q_{23}'-1=q_{33}q_{34}'-1=0$.
 Let $q:=q_{12}'$, $r:=q_{23}'$, and $\zeta:=q_{44}$.
 Then $(3)_{\zeta}=0$ and $q\notin \{-1,\zeta,\zeta^{-1}\}$.
 By $A^{r_2(X)}=B_4$ and $q\not=-1$,
 we get that $qr=1$.
 Hence by the assumption and~\cite[lemma 1.4]{W-17} we obtain the reflections of $X$
 \setlength{\unitlength}{1mm}
 \begin{align*}
X:~
\Dchainfour{4}{$-1$}{$q$}{$-1$}{$q^{-1}$}%
{$q$}{$q^{-1}$}{$\zeta$}
\quad \Rightarrow
r_4(X):~
\Dchainfour{3}{$-1$}{$q$}{$-1$}{$q^{-1}$}
{$\zeta q^{-1}$}{$ q\zeta^{-1}$}{$\zeta$}
\end{align*}
\begin{align*}
\quad \Rightarrow
r_3r_4(X):
\Drightofway{}{$-1$}{$q$}{-$\zeta q^{-2}$}{$(\zeta q)^{-1}$}{$\zeta^{-1}$}
{$\zeta q^{-1}$}{$q\zeta^{-1}$}{$\zeta$}
\end{align*}
 with $(3)_{\zeta}=0$ and $q\notin \{-1, \zeta, \zeta^{-1}\}$.
 Hence we get $-\zeta q^{-1}=1$ or $\zeta q^{-2} =1$ by~Lemma~\ref{le:Dynkin} on $A^{r_4(X)}=B_4$.
 Then $q=-\zeta$, $p\not=2,3$ or $q=-\zeta^{-1}$, $p\not=2,3$ by $q\notin \{-1, \zeta, \zeta^{-1}\}$.
 If $q=-\zeta^{-1}$ and $p\not=2,3$,
 then $\cC$ is standard and $\cD=\cD_{16,2}$.
 If $q=-\zeta$ and $p\not=2,3$,
 then $a_{24}^{r_3r_4(X)}=-2$,
 which is a contradiction.

  Subcase $b_{12}$.
 Consider that $q_{11}=q_{33}=-1$, $(3)_{q_{44}}=0$, $q_{12}'\not=-1$, $q_{44}^2 q_{34}'-1\not=0$, and $q_{22}q_{12}'-1=q_{22}q_{23}'-1=0$.
 Let $q:=q_{23}'$, $r:=q_{34}'$, and $\zeta:=q_{44}$.
 Then $q_{12}'=q \not=-1$.
 Then $qr=1$ by $A^{r_3(X)}=B_4$.
 Hence by~\cite[Lemma 1.4]{W-17}the reflection of $X$ is
 \setlength{\unitlength}{1mm}
 \begin{align*}
X:~
\Dchainfour{4}{$-1$}{$q$}{$q^{-1}$}{$q$}%
{$-1$}{$q^{-1}$}{$\zeta$}
\quad \Rightarrow \quad
r_4(X):~
\Dchainfour{}{$-1$}{$q$}{$q^{-1}$}{$q$}%
{$-\zeta q^{-2}$}{$ q\zeta^{-1}$}{$\zeta$}
\end{align*}
 with $(3)_{\zeta}=0$ and $q\notin\{-1, \zeta, \zeta^{-1}\}$.
 We obtain that $\zeta q^{-2}=1$ or $-\zeta q^{-2} \zeta^{-1} q= -\zeta q^{-2} q=1$ by~Lemma~\ref{le:Dynkin}$A^{r_4(X)}=B_4$.
 Then $q=-\zeta^{-1}$ and $p\not=2,3$ by $q\notin\{-1, \zeta, \zeta^{-1}\}$.
 Hence $\cC$ is standard and $\cD=\cD_{19,3}$.

 Subcase $b_{13}$.
 Consider that $q_{22}=q_{33}=-1$, $(3)_{q_{44}}=0$, $q_{44}^2 q_{34}'-1\not=0$, and $q_{11}q_{12}'-1=0$.
 In this case we assume that the condition $\{q_{12}',q_{23}'\}=\{-1\}$ doesn't hold and nor does the condition $\{q_{23},q_{34}'\}=\{-1\}$.
 Set $q:=q_{11}$, $r:=q_{12}'$, $s:=q_{34}'$, and $\zeta:=q_{44}$.
 Then $q=r=s^{-1}q\not=-1$ and $p\not=2$ by~\cite[Lemma 1.4]{W-17} on $A^{r_2(X)}=A^{r_3(X)}=B_4$.
 Then we get that $(3)_{\zeta}=0$ and $q\notin\{\zeta, \zeta^{-1}, -1\}$.
 Hence by~\cite[Lemma 1.4]{W-17} the reflection of $X$
 \setlength{\unitlength}{1mm}
 \begin{align*}
X:
\Dchainfour{4}{$q$}{$q^{-1}$}{$-1$}{$q$}%
{$-1$}{$q^{-1}$}{$\zeta$}
\quad \Rightarrow \quad
r_4(X):
\Dchainfour{}{$q$}{$q^{-1}$}{$-1$}{$q$}%
{$-\zeta q^{-2}$}{$ q\zeta^{-1}$}{$\zeta$}
\end{align*}
 implies that $\zeta q^{-2}=1$ or $-\zeta q^{-2} \zeta^{-1} q= -\zeta q^{-2} q=1$ by $A^{r_4(X)}=B_4$.
 Then $q=-\zeta^{-1}$ and $p\not=2,3$ by $q\notin \{-1,\zeta^{-1}\}$.
 Hence $\cC$ is standard and $\cD=\cD_{16,3}$.

Subcase $b_{14}$.
Consider that $q_{11}=q_{22}=q_{33}=-1$, $(3)_{q_{44}}=0$, $q_{12}'\not=-1$, and $q_{44}^2 q_{34}'-1\not=0$.
Assume that the condition $\{q_{23}',q_{34}'\}=\{-1\}$ doesn't hold in this case.
We get that $q_{12}'q_{23}'=q_{23}'q_{34}'=1$ by~\cite[Lemma 1.4]{W-17} on $A^{r_2(X)}=A^{r_3(X)}=B_4$.
Set $q:=q_{12}'$ and $\zeta:=q_{44}$.
Then $(3)_{\zeta}=0$ and $q\notin \{-1,  \zeta, \zeta^{-1}\}$.
By~\cite[Lemmas 1.3, 1.4]{W-17} and $A^{r_4(X)}=B_4$ the reflection of $X$
 \setlength{\unitlength}{1mm}
 \begin{align*}
X:~
\Dchainfour{4}{$-1$}{$q$}{$-1$}{$q^{-1}$}%
{$-1$}{$q$}{$\zeta$}
\quad \Rightarrow \quad
r_4(X):~
\Dchainfour{}{$-1$}{$q$}{$-1$}{$q^{-1}$}%
{-$\zeta q^{2}$}{$ (q\zeta)^{-1}$}{$\zeta$}
\end{align*}
implies that $q=-\zeta$ and $p\not=2,3$
Hence $\cD=\cD_{19,2}$.

 Case (c). Assume that $\cC(M)$ is standard of type $C_4$.
 By $a_{32}=-1$ and $a_{34}=-2$ we get that $q_{33}q_{23}'-1=(3)_{q_{33}}(q_{33}^2 q_{34}'-1)=0$ by~Lemma~\ref{le:Dynkin}.
 Further, by the assumption we obtain the equations $(2)_{q_{ii}}(q_{ii}q_{i,i+1}'-1)=(2)_{q_{jj}}(q_{jj}q_{j,j-1}'-1)=0$ for all $i\in \{1, 2\}$ and $j\in \{2, 4\}$,
 Here we distinguish two cases: $c_1$ and $c_2$.

 Subcase $c_1$.
 Assume that $$q_{33}^2 q_{34}'-1=q_{33}q_{23}'-1=(2)_{q_{ii}}(q_{ii}q_{i,i+1}'-1)=(2)_{q_{jj}}(q_{jj}q_{j,j-1}'-1)=0$$
 for all $i\in \{1, 2\}$ and $j\in \{2, 4\}$.
 Let $q:=q_{33}$.
 Then $q_{34}'=q^{-2}\not=1$ and $q_{23}'=q^{-1}\not=-1$.
 Hence $q_{22}q_{23}'=1$.
 Otherwise,
 if $q_{22}q_{23}'\not=1$
 then $q_{22}=-1$ and hence $a_{34}^{r_2(X)}=-1$, which is a contradiction.
 Then $q_{22}q_{12}'=1$ by $a_{21}=-1$.
 Hence $q_{11}q_{12}'=1$ and $q_{11}\not=-1$.
 Otherwise,
 if $q_{11}=-1$ then by~Lemma~\ref{le:Dynkin} and \cite[Lemma 1.4]{W-17} we obtain $a_{34}^{r_2 r_1(X)}=-1$,
 which is again a contradiction.
 If $q_{44}q_{34}'-1=0$,
 then $\cD=\cD_{31}$.
 If $q_{44}=-1$ and $q_{34}'\not=-1$,
 then the reflection of $X$
 \setlength{\unitlength}{1mm}
 \begin{align*}
X:~\Dchainfour{4}{$q$}{$q^{-1}$}{$q$}{$q^{-1}$}%
{$q$}{$q^{-2}$}{$-1$}
\quad \Rightarrow \quad
r_4(X):~
\Dchainfour{}{$q$}{$q^{-1}$}{$q$}{$q^{-1}$}%
{$-q^{-1}$}{$q^{2}$}{$-1$}
\end{align*}
 gives that $-q^{-1} q^{-1}=1$ by~Lemma~\ref{le:Dynkin} on $a_{32}^{r_4(X)}=-1$.
 Then $q_{34}'=q^{-2}=-1$, which is a contradiction.

 Subcase $c_2$.
 Assume that $$(3)_{q_{33}}=(2)_{q_{ii}}(q_{ii}q_{i,i+1}'-1)=(2)_{q_{jj}}(q_{jj}q_{j,j-1}'-1)=0$$
 for all $i\in \{1, 2\}$ and $j\in \{2, 4\}$.
 To avoid repetition we assume that the condition $q_{33}^2 q_{34}'-1\not=0$ holds.
 Set $\zeta:=q_{33}$.
 Then $q_{34}'\notin \{\zeta,\zeta^{-1}\}$.
 By $q_{33}q_{23}'-1=0$ we get $q_{23}'=\zeta^{-1}$.
 Since $A^{r_3}(X)=C_4$,
 we get that $q_{23}'q_{34}'=1$ and hence $q_{34}'=\zeta$,
 which is a contradiction.

 Case (d). Assume that $\cC(M)$ is standard of type $D_4$.
 By~Lemma~\ref{le:Dynkin} on $A^X=D_4$ we obtain $(2)_{q_{22}}(q_{22}q_{2i}'-1)=(2)_{q_{ii}}(q_{ii}q_{2i}'-1)=0$ for all $i\in \{1, 3, 4\}$,
 Then we distinguish three cases: $d_1$, $d_2$ and $d_3$.

 Subcase $d_1$.
 Assume that $q_{22}q_{2i}'-1=q_{ii}q_{2i}'-1=0$ for all $i\in \{1, 3, 4\}$.
 Then $\cD=\cD_{51}$.

 Subcase $d_2$.
 Assume that $q_{22}q_{2,i}'-1=0$ for all $i\in \{1, 3, 4\}$ and there exists $l\in \{1,3,4\}$ such that $(2)_{q_{ll}}=(2)_{q_{jj}}(q_{jj}q_{2j}'-1)=0$ for any $j\in \{1, 3, 4\} \setminus \{l\}$.
 Further, we assume that $q_{2l}'\not=-1$ to avoid the repetition.
 Set $q:=q_{22}$, $r:=q_{33}$ and $s:=q_{44}$.
 Then $q\not=-1$.
 By~\cite[Lemma 1.4]{W-17} the reflection of $X$
 \begin{align*}
X:
 \tau_{l2jk}\Dthreefork{l}{$-1$}{$q^{-1}$}{$q$}{$q^{-1}$}{$q^{-1}$}{$r$}{$s$}
 \quad \Rightarrow \quad
 r_l(X):
 \tau_{l2jk} \Dthreefork{}{$-1$}{$q$}{$-1$}{$q^{-1}$}{$q^{-1}$}{$r$}{$s$}
 \end{align*}
 with $j\not=k$ and $j, k\in \{1, 3, 4\}\backslash \{l\}$.
 Since $\cC$ is standard,
 we get that $q^{-2}=1$ by~\cite[Lemma 1.4]{W-17} on $A^{r_2r_l(X)}=D_4$,
 which is a contradiction.

 Subcase $d_3$.
 Assume that $(2)_{q_{22}}=(2)_{q_{ii}}(q_{ii}q_{2i}'-1)=0$ for all $i\in \{1, 3, 4\}$.
 Set $q:=q_{21}'$, $r:=q_{23}'$ and $s:=q_{24}'$.
 Assume that the condition $\{q,r,s\}=\{-1\}$ doesn't hold.
 Since $A^{r_2(X)}=D_4$,
 we get $q=r=s=-1$ or $qr=rs=qs=1$ by~\cite[Lemma 1.4]{W-17}.
 Hence $q=r=s=-1$,
 which is a contradiction.

 Case (f).
 Assume that $\cC(M)$ is standard of type $F_4$.
 By $a_{21}=-1$ and $a_{23}=-2$ we obtain $q_{22}q_{12}'-1=(3)_{q_{22}}(q_{22}^2 q_{23}'-1)=0$ by~Lemma~\ref{le:Dynkin}.
 Further, by~Lemma~\ref{le:Dynkin} on $A^X=F_4$ we get $(2)_{q_{ii}}(q_{ii}q_{i,i+1}'-1)=(2)_{q_{jj}}(q_{jj}q_{j,j-1}'-1)=0$ for all $i\in \{1, 3\}$ and $j\in \{2, 3, 4\}$,
 Then we distinguish two cases $f_1$ and $f_2$.

 Subcase $f_1$.
 Assume that
 $$q_{22}^2 q_{23}'-1=q_{22}q_{12}'-1=(2)_{q_{ii}}(q_{ii}q_{i,i+1}'-1)=(2)_{q_{jj}}(q_{jj}q_{j,j-1}'-1)=0$$
 for all $i\in \{1, 3\}$ and $j\in \{2, 3, 4\}$.
 Let $q:=q_{22}$.
 Then $q_{23}'=q^{-2}\not=1$ and $q_{12}'=q^{-1}\not=-1$ by $q_{22}q_{12}'-1=0$.
 Hence $q_{11}q_{12}'=1$ by $(2)_{q_{11}}(q_{11}q_{12}'-1)=0$.
 Otherwise,
 if $q_{11}=-1$ then by~\cite[Lemmas 1.3 and 1.4]{W-17} we get $a_{23}^{r_1(X)}=-1$,
 which is a contradiction.
 If $q_{33}=-1$,
 then $q_{22}^{r_3(X)}=-q^{-1}$ and $(q_{12}q_{21})^{r_3(X)}=q_{12}'= q^{-1}$ by~\cite[Lemma 1.4]{W-17}.
 Hence $q^2=-1$ by $a_{21}^{r_3(X)}=-1$.
 Then $q_{34}'=q_{44}=-1$ and hence $\cD=\cD_{41}$, where $q^2=-1$.
 If $q_{33}\not=-1$ and $q_{44}=-1$,
 then $q_{33}q_{23}'-1=q_{33}q_{34}'-1=0$ and $q_{33}=q^2\not=-1$.
 Then $q_{22}^{r_3r_4(X)}=-q^{-1}$ and $(q_{12}q_{21})^{r_3r_4(X)}=q^{-1}$ by~\cite[Lemma 1.4]{W-17}.
 Then by $a_{21}^{r_3r_4(X)}=-1$ we obtain that $q^2=-1$,
 which is a contradiction.
 If $q_{33}\not=-1$ and $q_{44}\not=-1$,
 Then $q_{44}=q^2$ and hence $\cD=\cD_{41}$.

 Subcase $f_2$.
 Assume that $$(3)_{q_{22}}=q_{22}q_{12}'-1=(2)_{q_{ii}}(q_{ii}q_{i,i+1}'-1)=(2)_{q_{jj}}(q_{jj}q_{j,j-1}'-1)=0$$
 for all $i\in \{1, 3\}$ and $j\in \{2, 3, 4\}$.
 Assume that the condition $q_{22}^2 q_{23}'-1\not=0$ holds to avoid the repetition.
 Let $\zeta:=q_{22}$.
 Then $q_{23}'=\zeta$.
 Since $A^{r_2(X)}=C_4$,
 which is a contradiction.
  \end{enumerate}
\end{proof}
\newpage


\setlength{\unitlength}{1mm}
\begin{table}
\centering
\begin{tabular}{r|l|l|l|}
row & gener. Dynkin diagrams & \text{fixed parameters} &\text{char} $\Bbbk$ \\
\hline \hline
 1 & \Dchainfour{}{$q$}{$q^{-1}$}{$q$}{$q^{-1}$}{$q$}{$q^{-1}$}{$q$}
& $q\in k^\ast \setminus \{1\}$ & $p>0$\\
\hline
 2 & \Dchainfour{}{$q^2$}{$q^{-2}$}{$q^2$}{$q^{-2}$}{$q^2$}{$q^{-2}$}{$q$}
& $q\in k^\ast $, $q^2\not=1$  & $p>0$\\
\hline
 3 & \Dchainfour{}{$q$}{$q^{-1}$}{$q$}{$q^{-1}$}{$q$}{$q^{-2}$}{$q^2$}
& $q\in k^\ast $, $q^2\not=1$  & $p>0$\\
\hline
 4 & \Dchainfour{}{$q^2$}{$q^{-2}$}{$q^2$}{$q^{-2}$}{$q$}{$q^{-1}$}{$q$}
& $q\in k^\ast $, $q^2\not=1$  & $p>0$\\
\hline
 5 & \Dthreefork{}{$q$}{$q^{-1}$}{$q$}{$q^{-1}$}{$q^{-1}$}{$q$}{$q$}
& $q\in k^\ast \setminus \{1\}$ & $p>0$\\
\hline
6 & \Dchainfour{}{$-1$}{$q^{-1}$}{$q$}{$q^{-1}$}{$q$}{$q^{-1}$}{$q$} &
$q\in k^\ast $, $q^2\not=1$  & $p>0$\\
  & \Dchainfour{}{$-1$}{$q$}{$-1$}{$q^{-1}$}{$q$}{$q^{-1}$}{$q$} & & \\
  & \Dchainfour{}{$q$}{$q^{-1}$}{$-1$}{$q$}{$-1$}{$q^{-1}$}{$q$} & & \\
\hline
7 & \Dchainfour{}{$-1$}{$q^{-2}$}{$q^2$}{$q^{-2}$}{$q^2$}{$q^{-2}$}{$q$} &
$q\in k^\ast$, $q^4\not=1$ & $p>0$\\
  & \Dchainfour{}{$-1$}{$q^2$}{$-1$}{$q^{-2}$}{$q^2$}{$q^{-2}$}{$q$} & & \\
  & \Dchainfour{}{$q^2$}{$q^{-2}$}{$-1$}{$q^2$}{$-1$}{$q^{-2}$}{$q$} & & \\
  & \Dchainfour{}{$q^2$}{$q^{-2}$}{$q^2$}{$q^{-2}$}{$-1$}{$q^2$}{$-q^{-1}$}& & \\
\hline
8 & \Dchainfour{}{$-1$}{$q^{-1}$}{$q$}{$q^{-1}$}{$q$}{$q^{-2}$}{$q^2$} &
$q\in k^\ast $, $q^2\not=1$ & $p>0$\\
  & \Dchainfour{}{$-1$}{$q$}{$-1$}{$q^{-1}$}{$q$}{$q^{-2}$}{$q^2$} & & \\
  & \Dchainfour{}{$q$}{$q^{-1}$}{$-1$}{$q$}{$-1$}{$q^{-2}$}{$q^2$} & & \\
  & \Drightofway{}{$q$}{$q^{-1}$}{$q$}{$q^{-1}$}{$q^{-1}$}{$-1$}{$q^2$}{$-1$} & & \\
\hline
9 & \Dchainfour{}{$q^2$}{$q^{-2}$}{$q^2$}{$q^{-2}$}{$q$}{$q^{-1}$}{$-1$} &
$q\in k^\ast $, $q^2,q^3\not=1$ & $p>0$\\
 & \Dchainfour{}{$q^2$}{$q^{-2}$}{$q^2$}{$q^{-2}$}{$-1$}{$q$}{$-1$} & & \\
& \Drightofway{}{$q^2$}{$q^{-2}$}{$-1$}{$q^2$}{$q^{-1}$}{$-1$}{$q^{-1}$}{$q$}\ \quad \quad
\Drightofway{}{$q^2$}{$q^{-2}$}{$-1$}{$q^2$}{$q$}{$-1$}{$q^{-3}$}{$-1$} & & \\
& \Dchainfour{}{$q^2$}{$q^{-2}$}{$q^2$}{$q^{-2}$}{$-1$}{$q^3$}{$q^{-3}$} & &  \\
& \Dchainfour{}{$q^2$}{$q^{-2}$}{$q$}{$q^{-1}$}{$-1$}{$q^3$}{$q^{-3}$} & &  \\
\hline
\end{tabular}
\end{table}

\newpage
\setlength{\unitlength}{1mm}
\begin{table}
\centering
\begin{tabular}{r|l|l|l|}
row & gener. Dynkin diagrams & fixed param. &\text{char} $\Bbbk$ \\
\hline \hline
10 & \Dchainfour{}{$q^{-1}$}{$q$}{$-1$}{$q^{-1}$}{$q$}{$q^{-1}$}{$q$} &
$q\in k^\ast $, $q^2\not=1$  & $p>0$ \\
 & \Dchainfour{}{$-1$}{$q^{-1}$}{$-1$}{$q$}{$-1$}{$q^{-1}$}{$q$} & &  \\
 & \Dchainfour{}{$-1$}{$q$}{$q^{-1}$}{$q$}{$-1$}{$q^{-1}$}{$q$} & &  \\
 & \Dchainfour{}{$-1$}{$q^{-1}$}{$q$}{$q^{-1}$}{$-1$}{$q$}{$-1$} & & \\
 & \Dchainfour{}{$-1$}{$q^{-1}$}{$q$}{$q^{-1}$}{$q$}{$q^{-1}$}{$-1$} & & \\
 & \Dchainfour{}{$-1$}{$q$}{$-1$}{$q^{-1}$}{$-1$}{$q$}{$-1$} & & \\
\hline
11 & \Dchainfour{}{$q^{-2}$}{$q^2$}{$-1$}{$q^{-2}$}{$q^2$}{$q^{-2}$}{$q$} &
$q\in k^\ast $, $q^4\not=1$ & $p>0$\\
 & \Dchainfour{}{$-1$}{$q^{-2}$}{$-1$}{$q^2$}{$-1$}{$q^{-2}$}{$q$} & & \\
 & \Dchainfour{}{$-1$}{$q^2$}{$q^{-2}$}{$q^2$}{$-1$}{$q^{-2}$}{$q$} & & \\
 & \Dchainfour{}{$-1$}{$q^{-2}$}{$q^2$}{$q^{-2}$}{$-1$}{$q^2$}{$-q^{-1}$} & & \\
 & \Dchainfour{}{$-1$}{$q^2$}{$-1$}{$q^{-2}$}{$-1$}{$q^2$}{$-q^{-1}$} & & \\
 & \Dchainfour{}{$q^2$}{$q^{-2}$}{$-1$}{$q^2$}{$q^{-2}$}{$q^2$}{$-q^{-1}$} & &  \\
\hline
12 & \Dchainfour{}{$q^{-1}$}{$q$}{$-1$}{$q^{-1}$}{$q$}{$q^{-2}$}{$q^2$} &
$q\in k^\ast $, $q^2\not=1$  & $p>0$\\
 & \Dchainfour{}{$-1$}{$q^{-1}$}{$-1$}{$q$}{$-1$}{$q^{-2}$}{$q^2$} & & \\
 & \Dchainfour{}{$-1$}{$q$}{$q^{-1}$}{$q$}{$-1$}{$q^{-2}$}{$q^2$} & &  \\
 & \Drightofway{}{$-1$}{$q^{-1}$}{$q$}{$q^{-1}$}{$q^{-1}$}{$-1$}{$q^2$}{$-1$}\  \quad \quad
\Drightofway{}{$-1$}{$q$}{$-1$}{$q^{-1}$}{$q^{-1}$}{$-1$}{$q^2$}{$-1$} & & \\
 & \Dthreefork{}{$q$}{$q^{-1}$}{$-1$}{$q$}{$q$}{$q^{-1}$}{$q^{-1}$} & &  \\
\hline
13 & \Dchainfour{}{$q$}{$q^{-1}$}{$q$}{$q^{-1}$}{$-1$}{$q^2$}{$q^{-2}$} &
$q\in k^\ast $, $q^2\not=1$  & $p>0$\\
 & \Drightofway{}{$q$}{$q^{-1}$}{$-1$}{$q$}{$q$}{$-1$}{$q^{-2}$}{$-1$}\  \quad \quad
\Dthreefork{}{$-1$}{$q$}{$-1$}{$q^{-1}$}{$q^{-1}$}{$q$}{$q$} & & \\
 & \Dthreefork{}{$-1$}{$q^{-1}$}{$q$}{$q^{-1}$}{$q^{-1}$}{$q$}{$q$} & & \\
\hline
\end{tabular}
\end{table}

\setlength{\unitlength}{1mm}
\begin{table}
\centering
\begin{tabular}{r|l|l|l|}
row & gener. Dynkin diagrams & fixed param. &\text{char} $\Bbbk$ \\
\hline \hline
14 & \Dchainfour{}{$q$}{$q^{-1}$}{$q$}{$q^{-1}$}{$-1$}{$-q$}{$-q^{-1}$} &
$q\in k^\ast $, $q^2\not=1$ & $p>2$ \\
 & \Drightofway{}{$q$}{$q^{-1}$}{$-1$}{$q$}{$-1$}{$-1$}{$-q^{-1}$}{$-1$}\   \quad \quad
\Drightofway{}{$-q^{-1}$}{$-q$}{$-1$}{$-q^{-1}$}{$-1$}{$-1$}{$q$}{$-1$} & & \\
 & \Dchainfour{}{$q$}{$q^{-1}$}{$-1$}{$-1$}{$-1$}{$-q$}{$-q^{-1}$} & & \\
 & \Dchainfour{}{$-q^{-1}$}{$-q$}{$-q^{-1}$}{$-q$}{$-1$}{$q^{-1}$}{$q$} & & \\
\hline
15 & \Dchainfour{}{$-\zeta ^{-1}$}{$-\zeta $}{$-\zeta ^{-1}$}%
{$-\zeta $}{$-\zeta ^{-1}$}{$-\zeta $}{$\zeta $}
& $\zeta \in G_3'$ & $p>3$\\
\hline
$15'$ & \Dchainfour{}{$-1$}{$-1$}{$-1$}%
{$-1$}{$-1$}{$-1$}{$1$}
&  & $p=3$ \\
\hline
16 & \Dchainfour{}{$-1$}{$-\zeta $}{$-\zeta ^{-1}$}{$-\zeta $}%
{$-\zeta ^{-1}$}{$-\zeta $}{$\zeta $} & $\zeta \in G_3'$& $p>3$ \\
  & \Dchainfour{}{$-1$}{$-\zeta ^{-1}$}{$-1$}{$-\zeta $}%
{$-\zeta ^{-1}$}{$-\zeta $}{$\zeta $} & & \\
  & \Dchainfour{}{$-\zeta ^{-1}$}{$-\zeta $}{$-1$}{$-\zeta ^{-1}$}%
{$-1$}{$-\zeta $}{$\zeta $} & &  \\
  & \Dchainfour{}{$-\zeta ^{-1}$}{$-\zeta $}{$-\zeta ^{-1}$}{$-\zeta $}%
{$-1$}{$-\zeta ^{-1}$}{$\zeta ^{-1}$} & & \\
\hline
17 & \Dchainfour{}{$-\zeta $}{$-\zeta ^{-1}$}{$-\zeta $}{$-\zeta ^{-1}$}%
{$-\zeta $}{$-\zeta ^{-1}$}{$\zeta $} & $\zeta \in G_3'$ & $p>3$\\
  & \Dchainfour{}{$-\zeta $}{$-\zeta ^{-1}$}{$-\zeta $}{$-\zeta ^{-1}$}%
{$-1$}{$-1$}{$\zeta $} & &  \\
  & \Drightofway{}{$-\zeta $}{$-\zeta ^{-1}$}{$-1$}{$-\zeta $}%
{$\zeta ^{-1}$}{$-1$}{$-1$}{$\zeta $} \quad \quad
 \Drightofway{}{$-\zeta $}{$-\zeta ^{-1}$}{$-1$}{$-\zeta $}%
{$\zeta $}{$-1$}{$-\zeta $}{$-1$} & &  \\
  & \Dchainfour{}{$-\zeta $}{$-\zeta ^{-1}$}{$\zeta $}{$\zeta ^{-1}$}%
{$-1$}{$-\zeta ^{-1}$}{$-\zeta $} & & \\
  & \Dchainfour{}{$-\zeta $}{$-\zeta ^{-1}$}{$-\zeta $}{$-\zeta ^{-1}$}%
{$-1$}{$-\zeta ^{-1}$}{$-\zeta $} & & \\
\hline
\end{tabular}
\end{table}


\begin{table}
\centering
\begin{tabular}{r|l|l|l|}
row & gener. Dynkin diagrams & fixed param. &\text{char} $\Bbbk$ \\
\hline \hline
18 & \Dchainfour{}{$\zeta ^{-1}$}{$\zeta $}{$\zeta ^{-1}$}{$\zeta $}%
{$\zeta $}{$\zeta ^{-1}$}{$-1$} & $\zeta \in G_3'$ & $p\not=3$\\
  & \Dchainfour{}{$\zeta ^{-1}$}{$\zeta $}{$\zeta ^{-1}$}{$\zeta $}%
{$-1$}{$\zeta $}{$-1$} & & \\
  & \Drightofway{}{$\zeta ^{-1}$}{$\zeta $}{$-1$}{$\zeta ^{-1}$}%
{$\zeta ^{-1}$}{$-1$}{$\zeta ^{-1}$}{$\zeta $} \quad \quad
  \Dthreefork{}{$\zeta ^{-1}$}{$\zeta $}{$-1$}{$\zeta ^{-1}$}%
{$\zeta $}{$-1$}{$-1$} & & \\
  & \Dthreefork{}{$\zeta ^{-1}$}{$\zeta $}{$\zeta $}{$\zeta ^{-1}$}%
{$\zeta ^{-1}$}{$-1$}{$-1$} \quad \quad
  \Dthreefork{}{$\zeta ^{-1}$}{$\zeta $}{$\zeta ^{-1}$}{$\zeta $}%
{$\zeta $}{$-1$}{$-1$} & & \\
\hline
19 & \Dchainfour{}{$-\zeta $}{$-\zeta ^{-1}$}{$-1$}{$-\zeta $}{$-\zeta ^{-1}$}%
{$-\zeta $}{$\zeta $} & $\zeta \in G_3'$ &$p>3$\\
  & \Dchainfour{}{$-1$}{$-\zeta $}{$-1$}{$-\zeta ^{-1}$}{$-1$}%
{$-\zeta $}{$\zeta $} & & \\
  & \Dchainfour{}{$-1$}{$-\zeta ^{-1}$}{$-\zeta $}{$-\zeta ^{-1}$}{$-1$}%
{$-\zeta $}{$\zeta $} & & \\
  & \Dchainfour{}{$-1$}{$-\zeta $}{$-\zeta ^{-1}$}{$-\zeta $}{$-1$}%
{$-\zeta ^{-1}$}{$\zeta ^{-1}$} & &  \\
  & \Dchainfour{}{$-1$}{$-\zeta ^{-1}$}{$-1$}{$-\zeta $}{$-1$}%
{$-\zeta ^{-1}$}{$\zeta ^{-1}$} & & \\
  & \Dchainfour{}{$-\zeta ^{-1}$}{$-\zeta $}{$-1$}{$-\zeta ^{-1}$}{$-\zeta $}%
{$-\zeta ^{-1}$}{$\zeta ^{-1}$} & & \\
\hline
20 & \Dchainfour{}{$\zeta ^{-1}$}{$\zeta $}{$-1$}{$\zeta ^{-1}$}{$\zeta $}%
{$\zeta $}{$-1$} & $\zeta \in G_3'$ &$p\not=3$ \\
 & \Dchainfour{}{$\zeta ^{-1}$}{$\zeta $}{$-1$}{$\zeta ^{-1}$}{$-\zeta ^{-1}$}%
{$\zeta ^{-1}$}{$-1$} & &\\
 & \Dchainfour{}{$-1$}{$\zeta ^{-1}$}{$-1$}{$\zeta $}{$-1$}%
{$\zeta $}{$-1$} & &\\
 & \Dchainfour{}{$-1$}{$\zeta $}{$\zeta ^{-1}$}{$\zeta $}{$-1$}%
{$\zeta $}{$-1$} & &\\
 & \Dchainfour{}{$-1$}{$\zeta ^{-1}$}{$-1$}{$\zeta $}{$\zeta $}%
{$\zeta ^{-1}$}{$-1$} & &\\
 & \Dchainfour{}{$-1$}{$\zeta $}{$\zeta ^{-1}$}{$\zeta $}{$\zeta $}%
{$\zeta ^{-1}$}{$-1$} & &\\
 & \Drightofway{}{$-1$}{$\zeta ^{-1}$}{$\zeta $}{$\zeta ^{-1}$}%
{$\zeta ^{-1}$}{$-1$}{$\zeta ^{-1}$}{$\zeta $}\ \quad \quad
 \Drightofway{}{$-1$}{$\zeta $}{$-1$}{$\zeta ^{-1}$}%
{$\zeta ^{-1}$}{$-1$}{$\zeta ^{-1}$}{$\zeta $} & &\\
 & \Dthreefork{}{$\zeta $}{$\zeta ^{-1}$}{$-1$}{$\zeta $}%
{$\zeta $}{$\zeta ^{-1}$}{$-1$}\ \quad \quad
 \Dthreefork{}{$\zeta $}{$\zeta ^{-1}$}{$\zeta $}{$\zeta $}%
{$\zeta ^{-1}$}{$\zeta ^{-1}$}{$-1$} & &\\
\hline
\end{tabular}
\end{table}

\begin{table}
\centering
\begin{tabular}{r|l|l|l|}
row & gener. Dynkin diagrams & fixed param.&\text{char} $\Bbbk$ \\
\hline \hline
21  & \Dchainfour{}{$-1$}{$\zeta ^{-1}$}{$\zeta $}{$\zeta ^{-1}$}%
{$\zeta $}{$\zeta $}{$-1$} & $\zeta \in G_3'$ &$p\not=3$ \\
  & \Dchainfour{}{$-1$}{$\zeta ^{-1}$}{$\zeta $}{$\zeta ^{-1}$}%
{$-\zeta ^{-1}$}{$\zeta ^{-1}$}{$-1$} & &\\
  & \Dchainfour{}{$-1$}{$\zeta $}{$-1$}{$\zeta ^{-1}$}{$\zeta $}%
{$\zeta $}{$-1$} & &\\
  & \Dchainfour{}{$-1$}{$\zeta $}{$-1$}{$\zeta ^{-1}$}{$-\zeta ^{-1}$}%
{$\zeta ^{-1}$}{$-1$} & &\\
  & \Dchainfour{}{$\zeta $}{$\zeta ^{-1}$}{$-1$}{$\zeta $}{$\zeta $}%
{$\zeta ^{-1}$}{$-1$} & & \\
  & \Dchainfour{}{$\zeta $}{$\zeta ^{-1}$}{$-1$}{$\zeta $}{$-1$}%
{$\zeta $}{$-1$} & &\\
  & \Drightofway{}{$\zeta $}{$\zeta ^{-1}$}{$\zeta $}{$\zeta ^{-1}$}%
{$\zeta ^{-1}$}{$-1$}{$\zeta ^{-1}$}{$\zeta $} & &\\
\hline
22 & \Dchainfour{}{$-\zeta $}{$\zeta $}{$-1$}{$-\zeta $}{$\zeta $}%
{$\zeta $}{$-\zeta $} & $\zeta \in G_4'$ &$p\not=2$ \\
 & \Dchainfour{}{$-1$}{$-\zeta $}{$-1$}{$\zeta $}{$-1$}{$\zeta $}{$-\zeta $}
& &\\
 & \Dchainfour{}{$-1$}{$\zeta $}{$-\zeta $}{$\zeta $}{$-1$}%
{$\zeta $}{$-\zeta $} & &\\
 & \Drightofway{}{$-1$}{$-\zeta $}{$\zeta $}{$-\zeta $}%
{$-1$}{$-1$}{$-\zeta $}{$-1$}\ \quad \quad
 \Drightofway{}{$-1$}{$\zeta $}{$-1$}{$-\zeta $}%
{$-1$}{$-1$}{$-\zeta $}{$-1$} & &\\
& \Dchainfour{}{$-1$}{$-\zeta $}{$\zeta $}{$-1$}{$-1$}%
{$\zeta $}{$-\zeta $} & &\\
& \Dchainfour{}{$-1$}{$\zeta $}{$-1$}{$-1$}{$-1$}{$\zeta $}{$-\zeta $} & &\\
& \Drightofway{}{$-\zeta $}{$\zeta $}{$-1$}{$-\zeta $}%
{$-1$}{$\zeta $}{$-\zeta $}{$-1$} & & \\
\hline
\end{tabular}
\caption{generalized Dynkin diagrams in all positive characteristic $p>0$}
\label{tab.1}
\end{table}

\setlength{\unitlength}{1mm}
\settowidth{\mpb}{$q_0\in k^\ast \setminus \{-1,1\}$,}
\rule[-3\unitlength]{0pt}{8\unitlength}
\begin{table}
\centering
\begin{tabular}{r|p{8.8cm}|l|l|}
 & \text{exchange graph}  &\text{row}&\text{char} $\Bbbk$\\
\hline \hline
 1 &
 \begin{picture}(2,2)
 \put(1,0){\scriptsize{$\cD_{11}$}}
 \end{picture}
&$1$ & $p>0$\\
\hline
 2 &\begin{picture}(2,4)
 \put(1,0){\scriptsize{$\cD_{21}$}}
 \end{picture}
& $2$ &$p>0$  \\
\hline
 3 &\begin{picture}(2,4)
 \put(1,0){\scriptsize{$\cD_{31}$}}
 \end{picture}
&$3$ &  $p>0$\\
\hline
4 &\begin{picture}(2,4)
 \put(1,0){\scriptsize{$\cD_{41}$}}
 \end{picture}
&$4$ &  $p>0$\\
\hline
 5 &\begin{picture}(2,4)
 \put(1,0){\scriptsize{$\cD_{51}$}}
 \end{picture}
&$5$ &  $p>0$\\
\hline
 6 &\begin{picture}(57,4.5)
 \put(1,0){\scriptsize{$\cD_{61}$}}
 \put(6,1){\line(1,0){7}}
 \put(9,2){\scriptsize{$1$}}
 \put(13,0){\scriptsize{$\cD_{62}$}}
 \put(19,1){\line(1,0){7}}
 \put(21,2){\scriptsize{$2$}}
 \put(26,0){\scriptsize{$\cD_{63}$}}
 \put(32,1){\line(1,0){6}}
 \put(34,2){\scriptsize{$3$}}
 \put(38,0){\scriptsize{$\tau_{4321} \cD_{62}$}}
 \put(50,1){\line(1,0){6}}
 \put(52,2){\scriptsize{$4$}}
 \put(56,0){\scriptsize{$\tau_{4321} \cD_{61}$}}
 \end{picture}
 &  $6$ &$p>0$\\
 \hline
 7  &\begin{picture}(60,4.5)
 \put(1,0){\scriptsize{$\cD_{71}$}}
 \put(6,1){\line(1,0){7}}
 \put(9,2){\scriptsize{$1$}}
 \put(13,0){\scriptsize{$\cD_{72}$}}
 \put(19,1){\line(1,0){7}}
 \put(21,2){\scriptsize{$2$}}
 \put(26,0){\scriptsize{$\cD_{73}$}}
 \put(32,1){\line(1,0){7}}
 \put(35,2){\scriptsize{$3$}}
 \put(39,0){\scriptsize{$\cD_{74}$}}
 \end{picture}
 & $7$ &$p>0$ \\
 \hline
 8 &
 \begin{picture}(60,15)
 \put(1,10){\scriptsize{$\cD_{81}$}}
 \put(6,11){\line(1,0){7}}
 \put(9,12){\scriptsize{$1$}}
 \put(13,10){\scriptsize{$\cD_{82}$}}
 \put(18,11){\line(1,0){6}}
 \put(21,12){\scriptsize{$2$}}
 \put(25,10){\scriptsize{$\cD_{83}$}}
 \put(30,11){\line(1,0){6}}
 \put(32,12){\scriptsize{$3$}}
 \put(37,10){\scriptsize{$\cD_{84}$}}
 \put(42,11){\line(1,0){8}}
 \put(46,12){\scriptsize{$4$}}
 \put(51,10){\scriptsize{$\tau_{1243}\cD_{83}$}}
 \put(51,0){\scriptsize{$\tau_{1243}\cD_{82}$}}
 \put(53,3){\line(0,1){5}}
 \put(55,4.5){\scriptsize{$2$}}
 \put(30,0){\scriptsize{$\tau_{1243}\cD_{81}$}}
 \put(43,1){\line(1,0){7}}
 \put(46,2){\scriptsize{$1$}}
 \end{picture}
 & $8$ &$p>0$\\
 \hline
9 &\begin{picture}(59,15)
 \put(1,10){\scriptsize{$\cD_{91}$}}
 \put(6,11){\line(1,0){7}}
 \put(9,12){\scriptsize{$4$}}
 \put(13,10){\scriptsize{$\cD_{92}$}}
 \put(18,11){\line(1,0){7}}
 \put(21,12){\scriptsize{$3$}}
 \put(26,10){\scriptsize{$\cD_{93}$}}
 \put(31,11){\line(1,0){7}}
 \put(34,12){\scriptsize{$2$}}
 \put(39,10){\scriptsize{$\tau_{3214}\cD_{96}$}}
 \put(51,11){\line(1,0){7}}
 \put(54,12){\scriptsize{$1$}}
 \put(59,10){\scriptsize{$\tau_{3214}\cD_{94}$}}
 \put(39,0){\scriptsize{$\tau_{3241} \cD_{95}$}}
 \put(42,3){\line(0,1){5}}
 \put(44,4.5){\scriptsize{$4$}}
 \end{picture}
 & $9$ &$p>0$\\
 \hline
10  &\begin{picture}(59,35)
 \put(1,30){\scriptsize{$\cD_{10,1}$}}
 \put(8,31){\line(1,0){7}}
 \put(11,32){\scriptsize{$2$}}
 \put(15,30){\scriptsize{$\cD_{10,2}$}}
 \put(22.5,31){\line(1,0){7}}
 \put(25.5,32){\scriptsize{$3$}}
 \put(30,30){\scriptsize{$\cD_{10,4}$}}
 \put(37.5,31){\line(1,0){8}}
 \put(41.5,32){\scriptsize{$4$}}
 \put(47,30){\scriptsize{$\cD_{10,5}$}}
 \put(15,20){\scriptsize{$\cD_{10,3}$}}
 \put(22.5,21){\line(1,0){7}}
 \put(25.5,22){\scriptsize{$3$}}
 \put(30,20){\scriptsize{$\cD_{10,6}$}}
 \put(37.5,21){\line(1,0){7}}
 \put(40.5,22){\scriptsize{$4$}}
 \put(45.5,20){\scriptsize{$\tau_{4321}\cD_{10,4}$}}
 \put(18,23){\line(0,1){5}}
 \put(19.5,24.5){\scriptsize{$1$}}
 \put(32.5,23){\line(0,1){5}}
 \put(33.5,24.5){\scriptsize{$1$}}
 \put(50,23){\line(0,1){5}}
 \put(51,24.5){\scriptsize{$1$}}
 \put(23,10){\scriptsize{$\tau_{3214}\cD_{10,3}$}}
 \put(37.5,11){\line(1,0){7}}
 \put(40.5,12){\scriptsize{$4$}}
 \put(45.5,10){\scriptsize{$\tau_{4321} \cD_{10,2}$}}
 \put(32.5,13){\line(0,1){5}}
 \put(33.5,14.5){\scriptsize{$2$}}
 \put(50,13){\line(0,1){5}}
 \put(51,14.5){\scriptsize{$2$}}
 \put(45.5,0){\scriptsize{$\tau_{4321} \cD_{10,1}$}}
 \put(50,3){\line(0,1){5}}
 \put(51,4.5){\scriptsize{$3$}}
 \end{picture}
 & $10$ &$p>0$\\
 \hline
11  &\begin{picture}(59,15)
 \put(1,10){\scriptsize{$\cD_{11,1}$}}
 \put(8,11){\line(1,0){7}}
 \put(11,12){\scriptsize{$2$}}
 \put(15,10){\scriptsize{$\cD_{11,2}$}}
 \put(22,11){\line(1,0){7}}
 \put(25,12){\scriptsize{$1$}}
 \put(29,10){\scriptsize{$\cD_{11,3}$}}
 \put(15,0){\scriptsize{$\cD_{11,4}$}}
 \put(17,3){\line(0,1){6}}
 \put(15,5){\scriptsize{$3$}}
 \put(29,0){\scriptsize{$\cD_{11,5}$}}
 \put(32,3){\line(0,1){6}}
 \put(30,5){\scriptsize{$3$}}
 \put(22,1){\line(1,0){7}}
 \put(25,2){\scriptsize{$1$}}
 \put(43,0){\scriptsize{$\cD_{11,6}$}}
 \put(36,1){\line(1,0){7}}
 \put(39,2){\scriptsize{$2$}}
 \end{picture}
 & $11$ &$p>0$\\
 \hline
12 &\begin{picture}(50,25)
 \put(10,20){\scriptsize{$\cD_{12,1}$}}
 \put(17.5,21){\line(1,0){7}}
 \put(20.5,22){\scriptsize{$2$}}
 \put(25,20){\scriptsize{$\cD_{12,2}$}}
 \put(33,21){\line(1,0){12}}
 \put(38.5,22){\scriptsize{$1$}}
 \put(47.5,20){\scriptsize{$\cD_{12,3}$}}
 \put(28,13){\line(0,1){5}}
 \put(29,14.5){\scriptsize{$3$}}
 \put(50,13){\line(0,1){5}}
 \put(51,14.5){\scriptsize{$3$}}
 \put(25,10){\scriptsize{$\cD_{12,4}$}}
 \put(33,11){\line(1,0){12}}
 \put(38.5,12){\scriptsize{$1$}}
 \put(47.5,10){\scriptsize{$\cD_{12,6}$}}
 \put(1,0){\scriptsize{$\tau_{1243}\cD_{12,1}$}}
 \put(14.5,1){\line(1,0){7}}
 \put(17.5,2){\scriptsize{$2$}}
 \put(23,0){\scriptsize{$\tau_{1243}\cD_{12,2}$}}
 \put(37.5,1){\line(1,0){7}}
 \put(40.5,2){\scriptsize{$1$}}
 \put(45.5,0){\scriptsize{$\tau_{1243} \cD_{12,3}$}}
 \put(28,3){\line(0,1){5}}
 \put(29,4.5){\scriptsize{$4$}}
 \put(50,3){\line(0,1){5}}
 \put(51,4.5){\scriptsize{$4$}}
 \end{picture}
 & $12$ &$p>0$\\
 \hline
13&\begin{picture}(59,15)
 \put(1,10){\scriptsize{$\cD_{13,1}$}}
 \put(8,11){\line(1,0){7}}
 \put(11,12){\scriptsize{$3$}}
 \put(15,10){\scriptsize{$\cD_{13,2}$}}
 \put(22,11){\line(1,0){7}}
 \put(25,12){\scriptsize{$2$}}
 \put(29,10){\scriptsize{$\cD_{13,4}$}}
 \put(36,11){\line(1,0){7}}
 \put(39,12){\scriptsize{$1$}}
 \put(43,10){\scriptsize{$\cD_{13,3}$}}
 \put(10,0){\scriptsize{$\tau_{1243}\cD_{11,1}$}}
 \put(17,3){\line(0,1){6}}
 \put(15,5){\scriptsize{$4$}}
 \end{picture}
 & $13$ &$p>0$\\
 \hline
14 &\begin{picture}(75,15)
 \put(10,10){\scriptsize{$\cD_{14,1}$}}
 \put(17.5,11){\line(1,0){7}}
 \put(20.5,12){\scriptsize{$3$}}
 \put(25,10){\scriptsize{$\cD_{14,2}$}}
 \put(33,11){\line(1,0){11}}
 \put(38.5,12){\scriptsize{$4$}}
 \put(45.5,10){\scriptsize{$\tau_{1243}\cD_{14,3}$}} \put(59.5,11){\line(1,0){7}}
 \put(62.5,12){\scriptsize{$2$}}
 \put(67.5,10){\scriptsize{$\tau_{3412}\cD_{14,5}$}}
 \put(1,0){\scriptsize{$\tau_{3214}\cD_{14,1}$}}
 \put(14.5,1){\line(1,0){7}}
 \put(17.5,2){\scriptsize{$1$}}
 \put(23,0){\scriptsize{$\tau_{3214}\cD_{14,2}$}}
 \put(37.5,1){\line(1,0){7}}
 \put(40.5,2){\scriptsize{$4$}}
 \put(45.5,0){\scriptsize{$\tau_{3241} \cD_{14,3}$}} \put(59.5,1){\line(1,0){7}}
 \put(62.5,2){\scriptsize{$2$}}
 \put(67.5,0){\scriptsize{$\tau_{1432} \cD_{14,5}$}}
 \put(28,3){\line(0,1){5}}
 \put(29,4.5){\scriptsize{$2$}}
 \end{picture}
 & $14$ &$p\not=2$\\
 \hline
15  &\begin{picture}(2,4)
 \put(1,0){\scriptsize{$\cD_{15,1}$}}
 \end{picture}
 & $15$ &$p\not=2,3$\\
 \hline
$15'$  &\begin{picture}(2,4)
 \put(1,0){\scriptsize{$\cD_{15',1}$}}
 \end{picture}
 & $15$ &$p=3$\\
 \hline
16 &
 \begin{picture}(60,5)
 \put(1,0){\scriptsize{$\cD_{16,1}$}}
 \put(9,1){\line(1,0){7}}
 \put(11,2){\scriptsize{$1$}}
 \put(17,0){\scriptsize{$\cD_{16,2}$}}
 \put(25,1){\line(1,0){7}}
 \put(29,2){\scriptsize{$2$}}
 \put(34,0){\scriptsize{$\cD_{16,3}$}}
 \put(42,1){\line(1,0){7}}
 \put(44,2){\scriptsize{$3$}}
 \put(49,0){\scriptsize{$\cD_{16,4}$}}
 \end{picture}
 & $16$ &$p\not=2,3$\\
 \hline
\end{tabular}
\end{table}

\setlength{\unitlength}{1mm}
\settowidth{\mpb}{$q_0\in k^\ast \setminus \{-1,1\}$,}
\rule[-3\unitlength]{0pt}{8\unitlength}
\begin{table}
\centering
\begin{tabular}{r|p{10.5cm}|l|l|}
(continued) & \text{exchange graph}  &\text{row}&\text{char} $\Bbbk$\\
\hline \hline
17 &\begin{picture}(50,25)
 \put(10,20){\scriptsize{$\cD_{17,2}$}}
 \put(17.5,21){\line(1,0){7}}
 \put(20.5,22){\scriptsize{$3$}}
 \put(25,20){\scriptsize{$\cD_{17,3}$}}
 \put(33,21){\line(1,0){11}}
 \put(38,22){\scriptsize{$2$}}
 \put(45.5,20){\scriptsize{$\tau_{3214}\cD_{17,6}$}}
 \put(60,21){\line(1,0){7}}
 \put(63,22){\scriptsize{$1$}}
 \put(68.5,20){\scriptsize{$\tau_{3214}\cD_{17,5}$}}
 \put(13,13){\line(0,1){5}}
 \put(15,14.5){\scriptsize{$4$}}
 \put(28,13){\line(0,1){5}}
 \put(29,14.5){\scriptsize{$4$}}
 \put(50,13){\line(0,1){5}}
 \put(51,14.5){\scriptsize{$4$}}
 \put(10,10){\scriptsize{$\cD_{17,1}$}}
 \put(23,10){\scriptsize{$\tau_{3421}\cD_{17,4}$}}
 \put(45.5,10){\scriptsize{$\tau_{3214}\cD_{17,4}$}}
 \put(68.5,10){\scriptsize{$\tau_{1432}\cD_{17,1}$}}
 \put(1,0){\scriptsize{$\tau_{3412}\cD_{17,5}$}}
 \put(15,1){\line(1,0){7}}
 \put(17.5,2){\scriptsize{$1$}}
 \put(23,0){\scriptsize{$\tau_{3412}\cD_{17,6}$}}
 \put(37.5,1){\line(1,0){7}}
 \put(40.5,2){\scriptsize{$4$}}
 \put(45.5,0){\scriptsize{$\tau_{1432} \cD_{17,3}$}}
 \put(60,1){\line(1,0){7}}
 \put(63,2){\scriptsize{$3$}}
 \put(68.5,0){\scriptsize{$\tau_{1432} \cD_{17,2}$}}
 \put(28,3){\line(0,1){5}}
 \put(29,4.5){\scriptsize{$2$}}
 \put(50,3){\line(0,1){5}}
 \put(51,4.5){\scriptsize{$2$}}
 \put(73.5,3){\line(0,1){5}}
 \put(75,4.5){\scriptsize{$2$}}
  \end{picture}
 & $17$ &$p\not=2,3$\\
 \hline
18 &\begin{picture}(80,25)
 \put(10,20){\scriptsize{$\cD_{18,1}$}}
 \put(17.5,21){\line(1,0){7}}
 \put(20.5,22){\scriptsize{$4$}}
 \put(25,20){\scriptsize{$\cD_{18,2}$}}
 \put(33,21){\line(1,0){14}}
 \put(40,22){\scriptsize{$3$}}
 \put(47.5,20){\scriptsize{$\cD_{18,3}$}}
 \put(55,21){\line(1,0){7}}
 \put(58,22){\scriptsize{$2$}}
 \put(63.5,20){\scriptsize{$\tau_{3214}\cD_{18,5}$}}
 \put(73.5,13){\line(0,1){5}}
 \put(74.5,14.5){\scriptsize{$4$}}
 \put(58,13){\line(2,1){10}}
 \put(65,14){\scriptsize{$1$}}
 \put(50,10){\scriptsize{$\tau_{3214}\cD_{18,6}$}}
 \put(68.5,10){\scriptsize{$\tau_{3214}\cD_{18,4}$}}
 \put(1,0){\scriptsize{$\tau_{3241}\cD_{18,1}$}}
 \put(15.5,1){\line(1,0){7}}
 \put(18.5,1.5){\scriptsize{$1$}}
 \put(23,0){\scriptsize{$\tau_{3241}\cD_{18,2}$}}
 \put(37.5,1){\line(1,0){7}}
 \put(40.5,1.5){\scriptsize{$4$}}
 \put(45.5,0){\scriptsize{$\tau_{3241} \cD_{18,3}$}}
 \put(60,1){\line(1,0){7}}
 \put(63,1.5){\scriptsize{$2$}}
 \put(68.5,0){\scriptsize{$\tau_{3241} \cD_{18,5}$}}
 \put(68.5,3.5){\line(-2,1){11}}
 \put(65,5.5){\scriptsize{$4$}}
 \put(73.5,3){\line(0,1){5}}
 \put(75,4.5){\scriptsize{$1$}}
  \end{picture}
 & $18$ &$p\not=3$\\
 \hline
19  &\begin{picture}(59,15)
 \put(1,10){\scriptsize{$\cD_{19,1}$}}
 \put(8,11){\line(1,0){7}}
 \put(11,12){\scriptsize{$2$}}
 \put(15,10){\scriptsize{$\cD_{19,2}$}}
 \put(22,11){\line(1,0){7}}
 \put(25,12){\scriptsize{$3$}}
 \put(29,10){\scriptsize{$\cD_{19,4}$}}
 \put(36,11){\line(1,0){7}}
 \put(39,12){\scriptsize{$2$}}
 \put(43,10){\scriptsize{$\cD_{19,6}$}}
 \put(15,0){\scriptsize{$\cD_{19,3}$}}
 \put(17,3){\line(0,1){6}}
 \put(15,5){\scriptsize{$1$}}
 \put(29,0){\scriptsize{$\cD_{19,5}$}}
 \put(32,3){\line(0,1){6}}
 \put(30,5){\scriptsize{$1$}}
 \put(22,1){\line(1,0){7}}
 \put(25,2){\scriptsize{$3$}}
 \end{picture}
 & $19$ &$p\not=2,3$\\
 \hline
20  &\begin{picture}(59,35)
 \put(1,30){\scriptsize{$\cD_{20,1}$}}
 \put(8,31){\line(1,0){7}}
 \put(11,32){\scriptsize{$4$}}
 \put(15,30){\scriptsize{$\cD_{20,2}$}}
 \put(22.5,31){\line(1,0){7}}
 \put(25.5,32){\scriptsize{$2$}}
 \put(30,30){\scriptsize{$\cD_{20,5}$}}
 \put(37.5,31){\line(1,0){8}}
 \put(41.5,32){\scriptsize{$1$}}
 \put(47,30){\scriptsize{$\cD_{20,6}$}}
 \put(30,20){\scriptsize{$\cD_{20,3}$}}
 \put(38,21){\line(1,0){8}}
 \put(41.5,22){\scriptsize{$1$}}
 \put(47,20){\scriptsize{$\cD_{20,4}$}}
 \put(30,22){\line(-1,1){7}}
 \put(24,23){\scriptsize{$2$}}
 \put(32.5,23){\line(0,1){5}}
 \put(33.5,24.5){\scriptsize{$4$}}
 \put(50,23){\line(0,1){5}}
 \put(51,24.5){\scriptsize{$4$}}
 \put(30,10){\scriptsize{$\cD_{20,7}$}}
 \put(37.5,11){\line(1,0){8}}
 \put(40.5,12){\scriptsize{$1$}}
 \put(47,10){\scriptsize{$\cD_{20,9}$}}
 \put(32.5,13){\line(0,1){5}}
 \put(33.5,14.5){\scriptsize{$3$}}
 \put(50,13){\line(0,1){5}}
 \put(51,14.5){\scriptsize{$3$}}
 \put(28,0){\scriptsize{$\cD_{20,10}$}}
 \put(37.5,1){\line(1,0){8}}
 \put(40.5,2){\scriptsize{$4$}}
 \put(47,0){\scriptsize{$\cD_{20,8}$}}
 \put(50,3){\line(0,1){5}}
 \put(51,4.5){\scriptsize{$2$}}
 \end{picture}
 & $20$ &$p\not=3$\\
 \hline
21  &\begin{picture}(59,15)
 \put(1,10){\scriptsize{$\cD_{21,1}$}}
 \put(8,11){\line(1,0){7}}
 \put(11,12){\scriptsize{$1$}}
 \put(15,10){\scriptsize{$\cD_{21,3}$}}
 \put(22,11){\line(1,0){7}}
 \put(25,12){\scriptsize{$2$}}
 \put(29,10){\scriptsize{$\cD_{21,6}$}}
 \put(36,11){\line(1,0){7}}
 \put(39,12){\scriptsize{$3$}}
 \put(43,10){\scriptsize{$\cD_{21,7}$}}
 \put(1,0){\scriptsize{$\cD_{21,2}$}}
 \put(8,1){\line(1,0){7}}
 \put(11,2){\scriptsize{$1$}}
 \put(3,3){\line(0,1){6}}
 \put(1,5){\scriptsize{$4$}}
 \put(15,0){\scriptsize{$\cD_{21,4}$}}
 \put(17,3){\line(0,1){6}}
 \put(15,5){\scriptsize{$4$}}
 \put(29,0){\scriptsize{$\cD_{21,5}$}}
 \put(32,3){\line(0,1){6}}
 \put(30,5){\scriptsize{$4$}}
 \put(22,1){\line(1,0){7}}
 \put(25,2){\scriptsize{$2$}}
 \end{picture}
 & $21$ &$p\not=3$\\
 \hline
22  &\begin{picture}(92,45)
 \put(1,40){\scriptsize{$\cD_{22,1}$}}
 \put(8,41){\line(1,0){7}}
 \put(11,42){\scriptsize{$2$}}
 \put(15,40){\scriptsize{$\cD_{22,2}$}}
 \put(22.5,41){\line(1,0){7}}
 \put(25.5,42){\scriptsize{$3$}}
 \put(30,40){\scriptsize{$\cD_{22,4}$}}
 \put(37.5,41){\line(1,0){7}}
 \put(40.5,42){\scriptsize{$4$}}
 \put(45.5,40){\scriptsize{$\tau_{1243}\cD_{22,5}$}}
 \put(15,30){\scriptsize{$\cD_{22,3}$}}
 \put(22.5,31){\line(1,0){7}}
 \put(25.5,32){\scriptsize{$3$}}
 \put(30,30){\scriptsize{$\cD_{22,8}$}}
 \put(37.5,31){\line(1,0){7}}
 \put(40.5,32){\scriptsize{$4$}}
 \put(45.5,30){\scriptsize{$\tau_{1243}\cD_{22,6}$}}
 \put(18,33){\line(0,1){5}}
 \put(19.5,34.5){\scriptsize{$1$}}
 \put(32.5,33){\line(0,1){5}}
 \put(33.5,34.5){\scriptsize{$1$}}
 \put(50,33){\line(0,1){5}}
 \put(51,34.5){\scriptsize{$1$}}
 \put(32.5,23){\line(0,1){5}}
 \put(33.5,24.5){\scriptsize{$2$}}
 \put(50,23){\line(0,1){5}}
 \put(51,24.5){\scriptsize{$2$}}
 \put(23,20){\scriptsize{$\tau_{3214}\cD_{22,7}$}}
 \put(45.5,20){\scriptsize{$\tau_{3412} \cD_{22,7}$}}
 \put(32.5,13){\line(0,1){5}}
 \put(33.5,14.5){\scriptsize{$4$}}
 \put(50,13){\line(0,1){5}}
 \put(51,14.5){\scriptsize{$4$}}
 \put(23,10){\scriptsize{$\tau_{1423}\cD_{22,6}$}}
 \put(37.5,11){\line(1,0){7}}
 \put(40.5,11.5){\scriptsize{$2$}}
 \put(45.5,10){\scriptsize{$\tau_{1432} \cD_{22,8}$}}
 \put(60,11){\line(1,0){7}}
 \put(63,11.5){\scriptsize{$3$}}
 \put(68.5,10){\scriptsize{$\tau_{1432} \cD_{22,3}$}}
 \put(32.5,3){\line(0,1){5}}
 \put(33.5,4.5){\scriptsize{$1$}}
 \put(50,3){\line(0,1){5}}
 \put(51,4.5){\scriptsize{$1$}}
 \put(75,3){\line(0,1){5}}
 \put(76,4.5){\scriptsize{$1$}}
 \put(23,0){\scriptsize{$\tau_{1423}\cD_{22,5}$}}
 \put(37.5,1){\line(1,0){7}}
 \put(40.5,1.5){\scriptsize{$2$}}
 \put(45.5,0){\scriptsize{$\tau_{1432} \cD_{22,4}$}}
 \put(60,1){\line(1,0){7}}
 \put(63,1.5){\scriptsize{$3$}}
 \put(68.5,0){\scriptsize{$\tau_{1432} \cD_{22,2}$}}
 \put(83,1){\line(1,0){7}}
 \put(86,1.5){\scriptsize{$4$}}
 \put(91,0){\scriptsize{$\tau_{1432} \cD_{22,1}$}}
 \end{picture}
 & $22$ &$p\not=2$\\
 \hline
\end{tabular}
\caption{The exchange graphs of $\cC(M)$ in Theorem~\ref{theo:clasi}.}
\label{tab.2}
\end{table}

\end{document}